\pdfoutput=1
\newif\ifpersonal
\documentclass[11pt,a4paper,dvipsnames,x11names]{amsart} 

\usepackage{amsmath,amsthm,amssymb,mathrsfs,mathtools,tensor,eucal} 
\usepackage[all,cmtip]{xy}

\usepackage[LGR,T1]{fontenc} 
\usepackage[utf8]{inputenc} 
\usepackage{CJKutf8}

\usepackage{xr}
\usepackage{appendix}

\usepackage{microtype,inconsolata} 
\usepackage{enumerate,comment,braket,xspace,tikz,tikz-cd,csquotes} 
\usetikzlibrary{shapes.geometric}
\usepackage[centering,vscale=0.7,hscale=0.7]{geometry}
\usepackage[hidelinks,colorlinks=true,linkcolor=Red4,citecolor=RoyalBlue]{hyperref}
\usepackage[capitalize]{cleveref}
\usepackage{xcolor}

\usepackage{mathpazo}
\usepackage{euler}
\usepackage{tikzarrows}

\tikzcdset{
	cells={font=\everymath\expandafter{\the\everymath\displaystyle}},
}

\numberwithin{equation}{subsection}
\theoremstyle{plain}
\newtheorem{theorem}[equation]{Theorem}
\newtheorem{thm}[equation]{Theorem}
\newtheorem{thm-intro}{Theorem}
\newtheorem{lemma}[equation]{Lemma}
\newtheorem{lem}[equation]{Lemma}
\newtheorem{proposition}[equation]{Proposition}
\newtheorem{prop}[equation]{Proposition}

\newtheorem{cor}[equation]{Corollary}
\newtheorem{corollary}[equation]{Corollary}
\newtheorem{cor-intro}[thm-intro]{Corollary}

\theoremstyle{definition}
\newtheorem{convention}[equation]{Convention}
\newtheorem{definition}[equation]{Definition}
\newtheorem{defin}[equation]{Definition}
\newtheorem{notation}[equation]{Notation}
\newtheorem{eg}[equation]{Example}
\newtheorem{eg-intro}[thm-intro]{Example}
\newtheorem{rem}[equation]{Remark}
\newtheorem{rem-intro}[thm-intro]{Remark}

\newtheorem{recollection}[equation]{Recollection}
\newtheorem{warning}[equation]{Warning}

\DeclareMathOperator{\Sing}{Sing}

\DeclareMathOperator{\cons}{cons}

\DeclareMathOperator{\bZ}{\mathbb{Z}}

\DeclareMathOperator{\at}{at}

\DeclareMathOperator{\Idem}{Idem}
\DeclareMathOperator{\Res}{Res}
\DeclareMathOperator{\hyp}{hyp}
\DeclareMathOperator{\Open}{Open}

\ifpersonal
\newcommand{\personal}[1]{\textcolor[rgb]{0,0,1}{(Personal: #1)}}
\newcommand{\discussion}[1]{\textcolor{violet}{(Discussion: #1)}}
\else
\newcommand{\personal}[1]{\ignorespaces}
\newcommand{\discussion}[1]{\ignorespaces}
\fi

\newcommand{\C}{\mathbb C}

\newcommand{\Q}{\mathbb Q}
\newcommand{\R}{\mathbb R}
\newcommand{\Z}{\mathbb Z}

\newcommand{\cA}{\mathcal A}
\newcommand{\cB}{\mathcal B}
\newcommand{\cC}{\mathcal C}
\newcommand{\cD}{\mathcal D}
\newcommand{\cE}{\mathcal E}

\newcommand{\cK}{\mathcal K}

\newcommand{\cM}{\mathcal M}

\DeclareFontFamily{U}{BOONDOX-calo}{\skewchar\font=45 }
\DeclareFontShape{U}{BOONDOX-calo}{m}{n}{<-> s*[1.05] BOONDOX-r-calo}{}
\DeclareFontShape{U}{BOONDOX-calo}{b}{n}{<-> s*[1.05] BOONDOX-b-calo}{}
\DeclareMathAlphabet{\mathcalboondox}{U}{BOONDOX-calo}{m}{n}

\newcommand{\bbT}{\mathbb T}

\makeatletter
\let\save@mathaccent\mathaccent
\newcommand*\if@single[3]{%
	\setbox0\hbox{${\mathaccent"0362{#1}}^H$}%
	\setbox2\hbox{${\mathaccent"0362{\kern0pt#1}}^H$}%
	\ifdim\ht0=\ht2 #3\else #2\fi
}
\newcommand*\rel@kern[1]{\kern#1\dimexpr\macc@kerna}
\newcommand*\widebar[1]{\@ifnextchar^{{\wide@bar{#1}{0}}}{\wide@bar{#1}{1}}}
\newcommand*\wide@bar[2]{\if@single{#1}{\wide@bar@{#1}{#2}{1}}{\wide@bar@{#1}{#2}{2}}}
\newcommand*\wide@bar@[3]{%
	\begingroup
	\def\mathaccent##1##2{%
		\let\mathaccent\save@mathaccent
		\if#32 \let\macc@nucleus\first@char \fi
		\setbox\z@\hbox{$\macc@style{\macc@nucleus}_{}$}%
		\setbox\tw@\hbox{$\macc@style{\macc@nucleus}{}_{}$}%
		\dimen@\wd\tw@
		\advance\dimen@-\wd\z@
		\divide\dimen@ 3
		\@tempdima\wd\tw@
		\advance\@tempdima-\scriptspace
		\divide\@tempdima 10
		\advance\dimen@-\@tempdima
		\ifdim\dimen@>\z@ \dimen@0pt\fi
		\rel@kern{0.6}\kern-\dimen@
		\if#31
		\overline{\rel@kern{-0.6}\kern\dimen@\macc@nucleus\rel@kern{0.4}\kern\dimen@}%
		\advance\dimen@0.4\dimexpr\macc@kerna
		\let\final@kern#2%
		\ifdim\dimen@<\z@ \let\final@kern1\fi
		\if\final@kern1 \kern-\dimen@\fi
		\else
		\overline{\rel@kern{-0.6}\kern\dimen@#1}%
		\fi
	}%
	\macc@depth\@ne
	\let\math@bgroup\@empty \let\math@egroup\macc@set@skewchar
	\mathsurround\z@ \frozen@everymath{\mathgroup\macc@group\relax}%
	\macc@set@skewchar\relax
	\let\mathaccentV\macc@nested@a
	\if#31
	\macc@nested@a\relax111{#1}%
	\else
	\def\gobble@till@marker##1\endmarker{}%
	\futurelet\first@char\gobble@till@marker#1\endmarker
	\ifcat\noexpand\first@char A\else
	\def\first@char{}%
	\fi
	\macc@nested@a\relax111{\first@char}%
	\fi
	\endgroup
}
\makeatother

\newcommand{\PSh}{\mathrm{PSh}}
\newcommand{\Sh}{\mathrm{Sh}}
\newcommand{\Shhyp}{\mathrm{Sh}^{\mathrm{hyp}}}
\newcommand{\HSh}{\mathrm{Sh}^{\mathrm{hyp}}}
\newcommand{\pHSh}{\tensor*[^{\mathfrak p}]{\mathrm{Sh}}{^{\mathrm{hyp}}}}
\newcommand{\ptauhypX}{\tensor*[^{\mathfrak p}]{\tau}{_{X}}} 
\newcommand{\rtauhypX}{\tensor*[^{\mathfrak r}]{\tau}{_{X}}} 

\newcommand{\Mod}{\mathrm{Mod}}

\newcommand{\Aff}{\mathrm{Aff}}

\newcommand{\Tor}{\mathrm{Tor}}

\newcommand{\Cat}{\categ{Cat}}
\newcommand{\dSet}{\categ{Set}} 
\newcommand{\cS}{\categ{Spc}}
\newcommand{\Spc}{\categ{Spc}}

\newcommand{\categ}[1]{\mathrm{#1}}

\newcommand{\PrL}{\categ{Pr}^{\mathrm{L}}}

\newcommand{\PrLomega}{\categ{Pr}^{\mathrm L,\omega}}

\newcommand{\Perf}{\mathrm{Perf}}
\newcommand{\bfPerf}{\mathbf{Perf}}

\newcommand{\bfCoh}{\mathbf{Coh}}

\DeclareMathOperator{\Cons}{Cons}
\newcommand{\ConsPhyp}{\Cons_P^{\mathrm{hyp}}}
\newcommand{\ConsRhyp}{\Cons_R^{\mathrm{hyp}}}
\newcommand{\ConsQhyp}{\Cons_Q^{\mathrm{hyp}}}
\newcommand{\ConsShyp}{\Cons_S^{\mathrm{hyp}}}
\newcommand{\Conshyp}{\Cons^{\mathrm{hyp}}}

\newcommand{\stack}[1]{\mathbf{#1}}
\newcommand{\bfConsPhyp}{\stack{Cons}_P}
\newcommand{\bfLoc}{\stack{LC}}
\newcommand{\bfConsRhyp}{\stack{Cons}_R}
\newcommand{\bfConsQhyp}{\stack{Cons}_Q}

\DeclareMathOperator{\LC}{LC}
\newcommand{\LChyp}{\LC^{\hyp}}

\newcommand{\st}{\mathrm{st}}

\newcommand{\inv}{^{-1}}
\newcommand{\id}{\mathrm{id}}

\newcommand{\op}{^\mathrm{op}}

\usetikzlibrary{decorations.markings} 
\tikzset{
  closed/.style = {decoration = {markings, mark = at position 0.5 with { \node[transform shape, xscale = .8, yscale=.4] {/}; } }, postaction = {decorate} },
  open/.style = {decoration = {markings, mark = at position 0.5 with { \node[transform shape, scale = .7] {$\circ$}; } }, postaction = {decorate} }
}

\DeclareMathOperator{\Fun}{Fun}
\DeclareMathOperator{\FunR}{Fun^R}

\DeclareMathOperator{\Hom}{Hom}

\DeclareMathOperator{\Map}{Map}
\DeclareMathOperator{\Mor}{Mor}

\DeclareMathOperator{\Spec}{Spec}

\DeclareMathOperator*{\colim}{colim}

\DeclareMathOperator*{\Piinfty}{\Pi_{\infty}}
\DeclareMathOperator{\pPervP}{{}^{\mathbf{\mathfrak p}} \mathbf{Perv}_P}
\DeclareMathOperator{\rPervR}{{}^{\mathbf{\mathfrak p}} \mathbf{Perv}_R}
\DeclareMathOperator{\pfrak}{\mathfrak p}
\DeclareMathOperator{\rfrak}{\mathfrak r}

\newcommand{\category}{$\infty$-category}

\newcommand{\Ind}{\mathrm{Ind}}

\def\noloc{{\mathpalette\reflectedop\colon}}
\def\reflectedop#1#2{\mathop{\reflectbox{$#1{#2}$}}}
\newcommand{\ptau}{{}^{\pfrak}\tau}
\newcommand{\dAff}{\mathrm{dAff}}
\newcommand{\bfCons}{\mathbf{Cons}}
\newcommand{\dCRing}{\mathrm{dCRing}}

\begin{document}
\title{The derived moduli of perverse sheaves}

\author{Peter J. Haine}
\address{Peter J. Haine, Department of Mathematics, University of Southern California, Kaprielian Hall, 3620 S Vermont Ave, Los Angeles, CA 90089, USA}
\email{phaine@usc.edu}

\author{Mauro PORTA}
\address{Mauro PORTA, Institut de Recherche Mathématique Avancée and Institut Universitaire de France (IUF), 7 Rue René Descartes, 67000 Strasbourg, France}
\email{porta@math.unistra.fr}

\author{Jean-Baptiste Teyssier}
\address{Jean-Baptiste Teyssier, Institut de Mathématiques de Jussieu and Institut Universitaire de France (IUF), 4 place Jussieu, 75005 Paris, France}
\email{jean-baptiste.teyssier@imj-prg.fr}

\date{\today}


\begin{abstract}
	We construct higher derived Artin stacks parametrizing constructible sheaves on complex algebraic varieties and compact real analytic varieties.
	Furthermore, we show that every perversity function gives rise to an open substack of perverse sheaves, which is a $1$-Artin stack locally of finite presentation that generalizes usual character stacks.
	As a sample application of the derived structure, we construct new examples of cohomological Hall algebras associated to punctured Riemann surfaces.
\end{abstract}

\maketitle


\tableofcontents

\section{Introduction}

This paper is concerned with the construction, in a high level of generality, of a derived Artin stack parametrizing perverse sheaves on a fixed stratified space satisfying some finiteness conditions.
The main results of this paper appeared first in a lesser generality in the first versions of \cite{Porta_Teyssier_Exodromy,Beyond_conicality}.
We subsequently decided to remove this material from the later versions of said papers and collect it here, in order to make them accessible to a wider community, without a heavy background in stratified homotopy theory.


\subsection*{Historical context}

Let $\Gamma$ be a finitely generated group.
The \emph{character variety of $\Gamma$} is a scheme $\mathbb X(\Gamma)$ parametrizing finite dimensional semi-simple representations of $\Gamma$.
Character varieties have a rich history; perhaps the first systematic treatment is \cite{Lubotzky_Magid}, although predated by work of Weil \cite{Weil_Discrete_subgroups_of_Lie_groups_I,Weil_Discrete_subgroups_of_Lie_groups_II}, Artin \cite{Artin_Azumaya} and Procesi \cite{Procesi_finite_dimensional} (see \cite{Sikora_Character_varieties} for a more recent treatment).
\personal{These references are not random: they are exactly the references cited in the introduction of \cite{Lubotzky_Magid}.}
They are ubiquitous in various areas of mathematics and physics, such as low-dimensional topology \cite{Bernard_Reidemeister_torsion_character_varieties,Gunningham_Jordan_Safronov}, gauge theory \cite{Atiyah_Bott_Yang_Mills,Hitchin_Self_duality}, mirror symmetry \cite{Hausel_Thaddeus}, nonabelian Hodge theory \cite{Simpson_Subspaces_of_moduli,Simpson_IHES_I,Simpson_IHES_II} and the geometric Langlands program \cite{BenZvi_Nadler_Betti_Langlands}.
As a result, their topology has been deeply investigated, especially in the case of surface groups (i.e., fundamental groups of complex algebraic curves).
They provide important examples of symplectic and Poisson varieties \cite{Goldman_Symplectic_nature,Goldman_Millson}, and their quantization has equally been studied in detail \cite{BenZvi_Brochier_Jordan,Ganev_Jordan_Safronov}.
The papers \cite{HLRV_I,HLRV_II} marked a landmark achievement in understanding their homology.

Character varieties can be seen as \emph{good moduli spaces} of more fundamental \emph{character stacks}.
These are Artin stacks parametrizing representations of $\Gamma$ together with their automorphisms; their singularities are milder than those of the character varieties.
Moreover, when $\Gamma = \pi_1(X,x)$ for a (sufficiently nice) topological space $X$, by the monodromy correspondence, character varieties can equally be described as the moduli of local systems on $X$.
In this case the \emph{higher homotopy groups} provide additional structure on the character variety: that of a \emph{derived scheme}.
This additional structure is a fundamental tool, heavily exploited in the geometric Langlands program; it also leads to a better understanding of the singularities of the character variety, and certain phenomena (e.g., their symplectic nature) becomes more transparent working simultaneously at the \emph{stacky and derived} level \cite{PTVV_2013}.
This is easy to see:

\begin{eg-intro}\label{eg:character_variety_S2}
	Since the $ 2 $-sphere $ S^2 $ is simply connected, the connected component of the character \emph{stack} of $ \pi_1(S^2) = \ast $ corresponding to representations of dimension $n$ coincides with $\mathrm{BGL}_n$ (the character variety is in this case reduced to a single point).
	On the other hand, the tangent complex of derived stack $\mathbf{LC}_n(S^2)$ at the trivial local system of rank $n$ is given by $\mathrm H^\ast_{\mathrm{sing}}(S^2;\Z)^{\oplus n}[1]$, see \cite[Proposition 2.2.6.6]{HAG-II}.
\end{eg-intro}

At the same time, local systems admit natural generalizations: constructible sheaves and perverse sheaves.
Combining ideas of stratified homotopy theory (that have been developed by the authors in \cite{Beyond_conicality} and recalled in \S\ref{sec:stratified_homotopy} in a way that they can be used as a black box) with the theory of moduli of flat objects of a category \cite{Toen_Moduli,Toen_Vaquie_Systemes_de_points,DPS_Stable_pairs}, in this paper, we construct both constructible and perverse generalizations of character stacks. 
We are only concerned with generalizing character \emph{stacks}; good moduli spaces for the perverse character stack has been obtained recently in \cite{Lampetti_Good_moduli}.
In the even more recent work \cite{Christ_Lampetti_Lagrangian_structures}, it is shown that the perverse character stack carries (under additional niceness hypotheses) a canonical shifted Poisson structure.
Both of these works should be understood as \emph{continuations and completions} of the current paper.


\subsection*{Moduli of perverse sheaves}

In \cite{Gelfand_MacPherson_Vilonen}, the authors attached to each complex projective variety $X$ equipped with a Whitney stratification $P$ a finite quiver $Q(X,P)$ \emph{with relations}, whose representations coincide with perverse sheaves on $(X,P)$, with respect to the middle perversity.
Out of this, it is easy to construct an Artin stack $\mathbf{Rep}_{Q(X,P)}$ of representations of $Q(X,P)$, which can be interpreted as a moduli of perverse sheaves.
Before going any further, let us explain the two main limitations of this approach:


\subsubsection*{Poor functoriality} 

The quiver $Q(X,P)$ depends on several choices (and most notably a projective embedding of $X$ chosen as a function of the stratification $P$).
This makes it hard to control how $\mathbf{Rep}_{Q(X,P)}$ depends on the stratified $(X,P)$ itself.
For instance, it is not completely clear from this perspective how to let $P$ vary. 
Instead, our method allows us to construct a moduli stack of sheaves that are perverse with respect to \emph{some} stratification, see \cref{thm_intro:varying_stratification} below.
In a similar spirit, in \cite{Nitsure_Sabbah_I,Nitsure_Regular_holonomic} the authors deliberately chose an alternative presentation of perverse sheaves because the quiver of \cite{Gelfand_MacPherson_Vilonen} did not interact well with the Riemann--Hilbert correspondence.


\subsubsection*{Absence of a derived structure} 

Recall from \cite[\S I.5.5]{Porta_HDR} (or \cite[\S II.2.6]{DPS_Stable_pairs}) that a standard way to produce a derived enhancement of the moduli of objects of an \emph{abelian} category $\cA$ is to exhibit $\cA$ as the heart of a sufficiently nice $t$-structure $\tau$ on a stable $\infty$-category $\cC$ \emph{of finite type}.
In this case, the finitness allows us to form Toën and Vaquié's moduli of object construction $\cM_\cC$.
The $t$-structure allows us to cut out an open substack $\bfCoh_{\mathsf{ps}}(\cC,\tau)$ of $\cM_\cC$, which is the desired enhancement of the moduli of objects of $\cA = \cC^\heartsuit$. 
Notice that the derived enhancement \emph{depends} on the choice of $(\cC,\tau)$.

In the case at hand, the natural choice would be the derived category of representations of $Q(X,P)$, but since this quiver has relations, its derived category is not generally of finite type in the sense of \cite{Toen_Moduli}:

\begin{eg-intro}
	Let $X = \mathbb P^2_{\mathbb C}$ be the complex projective plane, stratified in a copy of $\mathbb P^1_{\mathbb C}$.
	It is shown in \cite[Example 6.3]{MacPherson_Vilonen_Elementary_construction} that the category of perverse sheaves in this case is equivalent to the category of representations of the quiver
	\[ \begin{tikzcd}
		\bullet \arrow[shift left=3pt]{r}{a} & \bullet \arrow[shift left=3pt]{l}{b} \ ,
	\end{tikzcd} \]
	subject to the relations $ab = ba = 0$.
	In other words, this category of perverse sheaves is the category of representations of the preprojective algebra of the quiver $A_2$, which has infinite homological dimension -- see e.g., \cite[Remark 4.11]{Brenner_Butler_King}.
\end{eg-intro}

For this special example, it is easy to find a good replacement, namely using the $2$-Calabi-Yau completion (a.k.a.\ Ginzburg dg-algebra) of $A_2$ \cite{Keller_Deformed_CY,Bozec_Calaque_Scherotzke_Relative}) instead of its preprojective algebra.
For more complicated stratified spaces, we would not know how to proceed in this direction.
Even worse, \cref{eg:character_variety_S2} shows that even when $Q(X,P)$ is homologically of finite type, the resulting derived structure does not match the naturally expected one.

Instead, our method provides a canonical derived structure, as a concrete incarnation of the geometry of the stratification: the tangent complex of our derived stack is highly sensitive to the cohomology of the \emph{strata and the links} of $(X,P)$. 
In this approach, the homological finitess property of $Q(X,P)$ is replaced by the finiteness theorems of the exit-path $\infty$-categories proved in \cite{Beyond_conicality}.
The construction of the perverse cohomological Hall algebra (see \cref{thm:perverse_CoHA}) is a first piece of evidence that the derived structure we construct here is very robust. 
Another piece of evidence has been recently obtained by Christ--Lampetti \cite{Christ_Lampetti_Lagrangian_structures}, where they show that the derived stack of perverse sheaves constructed here has a canonical shifted Poisson structure.

Our approach has also some other \emph{minor-but-pleasant} features: dropping the projectivity and Whitney assumptions on $(X,P)$, and the restriction to middle perversity.


\subsection*{Moduli of constructible complexes}

\personal{(Mauro) I think that below it is important to never mention the word hypersheaf, which is essentially a turn-off for anyone vaguely interested in the geometry of this moduli space.}
Central to our approach is the construction of a bigger, highly non-separated higher Artin stack in the sense of Simpson \cite{Simpson_Algebraic_1996}.
This larger derived stack, denoted $\bfConsPhyp(X)$, parametrizes \emph{constructible complexes} and is obtained as a \emph{moduli of objects} in the sense of Toën--Vaquié \cite{Toen_Moduli}.
A subtle point that deserves to be mentioned here is that $\bfConsPhyp(X)$ is realized as the moduli of objects of the \emph{large} stable $\infty$-category $\mathrm{Cons}_P(X;\Mod_k)$ of \emph{unbounded constructible complexes with infinite dimensional stalks}.
The key insight is that the finiteness property of the exit-path category $\Pi_\infty(X,P)$ (proven in a large variety of examples in \cite{Beyond_conicality}) allows to show that $\mathrm{Cons}_P(X;\Mod_k)$, which a priori is very large, is in fact of finite type in the sense of \cite{Toen_Moduli}.
This allows us to bootstrap off the main theorem of \emph{loc.\ cit.}
One of the fundamental points is that the stalk functors are corepresentable in $\mathrm{Cons}_P(X;\Mod_k)$ by objects that are compact, but with infinite dimensional stalks (see \cref{obs:objects_of_Exit_with_locally_weakly_contractible_strata} and \cite[Warning 6.5.2]{Porta_Teyssier_Exodromy}).
This implies that the moduli of objects of $\mathrm{Cons}_P(X;\Mod_k)$ (whose construction is recalled in \cref{recollection_Toen_Vaquie}) actually parametrizes $P$-constructible sheaves with \emph{perfect} stalks.

As several papers already demonstrated \cite{Artin_Zhang,Toen_Vaquie_Systemes_de_points,BLMNPS,DPS_Stable_pairs}, a sufficiently well-behaved $t$-structure $\tau$ allows one to cut out a better behaved open $1$-Artin substack inside of $\bfCons_P(X)$, parametrizing families of $\tau$-flat objects.
Indeed, the bulk of this paper is devoted to check that the perverse $t$-structure (of any perversity) is well-behaved in this sense.


\subsection*{Statement of results}

We can now state our main results.
The first main theorem below is a consequence of \cref{prop:categorically_compact_implies_representable} and the finiteness results of \cite{Beyond_conicality}, recalled in the main text as Theorems~\ref{thm:subanalytic_stratified_spaces_are_conically_refineable} \& \ref{thm:algebraic_stratified_spaces_are_conically_refineable_and_categorically_finite}.
The statements involve \emph{subanalytic stratified spaces} (in the sense of \cref{def:subanalytic_stratified_space}); the key example of such is a locally finite stratification a of real analytic manifold by subanalytic subsets.
We also use the term \emph{algebraic stratified space} to mean a stratification of the $ \mathbb{R} $-points of an $ \mathbb{R} $-variety by Zariski locally closed subsets (see \cref{def:algebraic_stratified_space}).

\personal{(Mauro): I'd avoid to state the most general result here - for $(X,P)$ categorically compact and locally categorically compact. The reason is that excess of generality might end up being a turn-off for the reader.}

\begin{thm-intro}[Fixed stratification, see \cref{prop:categorically_compact_implies_representable} \& \cref{cor:examples_when_ConsP_is_locally_geometric}]\label{thm-intro:perverse}
	Let $(X,P)$ be either a compact subanalytic stratified space or an algebraic stratified space, and let $ k $ be a connective derived commutative ring (a.k.a. an animated commutative ring or simplicial commutative ring).
	Then there exists a geometric derived stack $\bfCons_P(X)$ locally of finite presentation over $k$ parametrizing complexes of $P$-constructible sheaves with perfect stalks.
	Furthermore, for every perversity function $\pfrak \colon P \to \mathbb Z$, there exists an open $1$-Artin substack
	\begin{equation*}
		\pPervP(X) \subset \bfCons_P(X)
	\end{equation*}
	parametrizing flat families of perverse sheaves on $X$.
\end{thm-intro}

\begin{thm-intro}[Varying stratification, see \cref{thm:varying_stratification} \& \cref{cor:varying_stratification_perverse}]\label{thm_intro:varying_stratification}
	Let $X$ be a compact subanalytic space (resp., a $\mathbb R$-algebraic variety) and let $ k $ be a connective derived commutative ring.
	Then there exists a geometric derived stack $\bfCons(X)$ locally of finite presentation over $k$ parametrizing complexes of sheaves with perfect stalks which are constructible with respect to \emph{some} subanalytic (resp., $\mathbb R$-algebraic) stratification.
	If $\pfrak$ denotes the middle perversity function, there is an open $1$-Artin substack 
	\begin{equation*}
		\tensor*[^{\pfrak}]{\mathbf{Perv}}{}(X) \subset \bfCons(X)
	\end{equation*}
	parametrizing flat families of perverse sheaves that are constructible with respect to \emph{some} subanalytic (resp., $\mathbb R$-algebraic) stratification.
\end{thm-intro}

These result are special cases of more general ones on stratified spaces satisfying appropriate finiteness conditions; see \cref{prop:categorically_compact_implies_representable} and \cref{thm:representability_perv}.

Finally we showcase the utility of the derived structure by establishing the existence of a \emph{perverse cohomological Hall algebra}, see \cref{thm:perverse_CoHA}.
It should be noted that this last result is a straightforward consequence of the main results of \cite{DPS_Stable_pairs}, the criterion for the properness of quot schemes obtained in \cite{Lampetti_Good_moduli} and the formula for the cotangent complex obtained in \cref{prop:categorically_compact_implies_representable}.

%
%


\subsection*{Notation}

All derived commutative rings (in the sense of Raksit \cite[\S4]{arXiv:2007.02576}) considered in this paper are \textit{connective}, i.e., are what are typically called animated commutative rings or simplicial commutative rings.
Therefore, in this paper, we simply use the term \textit{derived commutative ring} to refer to a \textit{connective} derived commutative ring, and write $\dCRing$ for the $\infty$-category of connective derived commutative rings.
For $ A \in \dCRing $, we write $ \Mod_A $ for the stable $\infty$-category of $A$-modules.
We write $\Perf_A \subset \Mod_A$ for the full subcategory spanned by \textit{perfect} $A$-modules, i.e., the smallest stable full subcategory containing $A$ and closed under retracts.
The $\infty$-category $ \Mod_A $ is compactly generated with full subcategory of compact objects $ \Perf_A $ \cite[Proposition 7.2.4.2]{Lurie_Higher_algebra}, \cite[Notation 25.2.1.1]{Lurie_SAG}.

Let $ \cC $ and $ \cD $ be presentable $ \infty $-categories.
We write $ \cC \otimes \cD $ for the \emph{tensor product} of $ \cC $ and $ \cD $ in the $ \infty $-category $ \PrL $ of presentable $ \infty $-categories and left adjoint functors.
The tensor product is the universal presentable $ \infty $-category equipped with a functor $ \cC \times \cD \to \cC \otimes \cD $ that preserves colimits separately in each variable. 
We refer the reader to \cite[\S4.8.1]{Lurie_Higher_algebra} for further background in the tensor product.


\subsection*{Acknowledgments}

We are grateful to Federico Binda, Bernhard Keller, Enrico Lampetti, Nitin Nitsure, and Olivier Schiffmann for useful conversations about this paper.

PH gratefully acknowledges support from the NSF Mathematical Sciences Postdoctoral Research Fellowship under Grant \#DMS-2102957.
MP and JBT gratefully acknowledge the support of the ANR grant DAG-Arts (24-CE40-4098) and of the Institut Universitaire de France (IUF).


\section{Recollections on stratified homotopy theory}\label{sec:stratified_homotopy}

In this section, we collect some general material on (hyper)constructible (hyper)sheaves and exodromic stratifications.
These results are not original: most of them were obtained in \cite[Appendix A]{Lurie_Higher_algebra} or in \cite{Porta_Teyssier_Exodromy,HPT,Beyond_conicality}.
The reader already acquainted with these ideas is invited to skip this section.


\subsection{Stratified topological spaces}

Let $P$ be a poset.
We equip the underlying set of $ P $ with the topology whose open subsets are the upward-closed subsets $Q\subset P$.
A \emph{stratified space} is a topological space $X$ equipped with a continuous map to a poset.
For a subset $S\subset P$, we set $X_S \coloneqq S \times_P X$ and we let $X_S\to S$ be the induced stratification; we denote by $i_S \colon X_S \hookrightarrow X$ the inclusion.
For $a\in P$, the subset $X_a $ is the \emph{stratum of $(X,P)$ over $a$}.
The collection of stratified spaces organize into a category in an obvious manner.

\begin{rem}
	We abuse notation by denoting a stratification of $X$ by $P$ as $(X,P)$ instead of $ X\to P$ and refer to $(X,P)$ as a stratified space.
\end{rem}

\begin{eg}
	Let $X$ be a topological space equipped with a partition $S = \{X_p\}_{p \in P}$ into disjoint locally closed subspaces.
	We define a partial order on $P$ by $p \leqslant q $ if and only if $ X_p \cap \overline{X_q} \ne \emptyset$.
	This gives rise to a continuous map $X \to P$, which is a stratification in the above sense.
\end{eg}

\begin{eg}
	Let $f \colon Y\to Q$ be a stratified space.
	Write $C(Y) \coloneqq \ast \sqcup \left( Y\times \mathbb{R}_{>0} \right)$.
	The set $C(Y)$ is endowed with the topology whose open subsets are the subsets $U\subset C(Y)$ such that $U\cap \left( Y\times \mathbb{R}_{>0}\right)$ is open and if $\ast \in U$, then 
	\[ C_{\varepsilon}(Y)\coloneqq \{\ast\} \sqcup \left( Y\times (0,\varepsilon) \right)\subset U \]
	for some $\varepsilon >0$.
	Let $Q^{\lhd} $ be the poset obtained from $Q$ by adding a smallest element $-\infty$.
	We define a continuous map $g \colon C(Y)\to Q^{\lhd}$ by sending $\ast$ to $-\infty$  and $(y,t)\in Y\times \mathbb{R}_{>0} $ to $f(y)$.
	We refer to $(C(Y),Q^{\lhd})$ as the \textit{cone} of $(Y,Q)$.
\end{eg}

The following should be thought as a (purely topological) regularity condition on a stratified space (akin to the notion of a Whitney stratification).

\begin{defin}\label{def_conical}
	Let $(X,P)$ be a stratified space.
	We say that $(X,P)$ is \textit{conically stratified} if for every point $x\in X$ lying over $a\in P$, there exists a topological space $Z$, a stratified space $(Y,P_{> a})$ and a map of stratified spaces  $(Z \times C(Y), P_{\geqslant a}) \to (X,P)$ inducing an homeomorphism between  $Z \times C(Y)$ and an open neighbourhood of $x$ in $X$.
\end{defin}

\begin{definition}
	Let $ X $ be a topological space.
	We say that a stratification $ X\to P $ is \textit{locally finite} if for every point $ x \in X $, there is an open neighborhood $ U $ of $ x $ such that $ U $ intersects only finitely many strata of $ (X,P) $.
\end{definition}

\begin{eg}[Simplicial complexes]\label{eg:simplicial_complexes}
	Let $ (V,S) $ be a simplicial complex, and regard $ S $ as a poset ordered by inclusion.
	Write $ \Delta^{(V,S)} $ for the geometric realization of $ (V,S) $.
	There is a natural stratification $ \Delta^{(V,S)}\to S $ with contractible strata, see \cite[Definition A.6.7]{Lurie_Higher_algebra}.
	Moreover, \cite[Proposition A.6.8]{Lurie_Higher_algebra} shows that if $(V,S)$ is \textit{locally finite}, this is a conical stratification.
\end{eg}

\begin{eg}
	If $X$ is a differentiable manifold equipped with a \emph{Whitney stratification} $P$ (see e.g., \cite[Definition 2.5]{Nocera_Volpe_Whitney_stratifications}), then it is conically stratified.
	In fact, it is even \emph{conically smooth} in the sense of \cite{Ayala_Francis_Tanaka_Local_structures}, see \cite{Nocera_Volpe_Whitney_stratifications}.
\end{eg}

We conclude by singling out two classes of stratified spaces that are particularly important in this paper.

\begin{definition}\label{def:subanalytic_stratified_space}
	A \textit{subanalytic stratified space} is the data of a triple $ (M,X,P) $ where $ M $ is a smooth $\R$-analytic space, $ X \subset M $ is a locally closed subanalytic subset, and $ X \to P $ is a locally finite stratification by subanalytic subsets of $ M $.
\end{definition}

\begin{definition}\label{def:algebraic_stratified_space}
	An \textit{algebraic stratified space} is the data of a stratified space $ (X,P) $ where $ X $ is (the real points of) an algebraic variety over $ \R $ and $ X \to P $ is a finite stratification by Zariski locally closed subsets.
\end{definition}


\subsection{Hypersheaves, hyperconstancy, and hyperconstructibility}

We content ourselves to introduce our notation, and we refer to \cite[\S1.1]{HPT}, \cite[\S1.2]{Beyond_conicality}, or \cite[\S III.2.1]{Porta_HDR} for a more thorough review of these notions.

\begin{notation}
	Let $X$ be a topological space.
	We denote by $\PSh(X)$ the $\infty$-category of $\Spc$-valued presheaves on $X$ and by $\Sh(X)$ (resp., $\HSh(X)$) the full subcategory spanned by sheaves (resp., hypersheaves).
	For $\cE \in \PrL$, we set
	\[ \HSh(X;\cE) \coloneqq \HSh(X)\otimes \cE\simeq \FunR(\HSh(X)\op,\cE)  \ , \]
	and similarly for (pre)sheaves.
\end{notation}

\begin{rem}
	Sheaves and hypersheaves on topological spaces coincide in many situations that are the central focus of interest of the current paper. For example, this is the case if $X$ admits a CW structure \cite{MO:168526}.
	The reader uneasy with hypersheaves and only interested in reasonable geometric examples can therefore simply ignore the attribute \emph{hyper} in the rest of the paper.
\end{rem}

\begin{rem}
	The functoriality of the tensor product in $\PrL$ induces for every $\cE \in \PrL$ a \emph{hypersheafification functor} $(-)^{\hyp} \colon \PSh(X;\cE)\to \HSh(X;\cE)$, which is still a localization by \cite[Lemma I.4.11]{Porta_HDR}.
\end{rem}

\begin{notation}
	Let $f \colon X \to Y$ be a map of topological spaces and let $\cE \in \PrL$.
	We denote by
	\[ f\inv \colon \PSh(Y;\cE) \leftrightarrows \PSh(X;\cE) \colon f_\ast \]
	the induced adjunction.
	The functor $f_\ast$ preserves (hyper)sheaves, and it is part of an adjunction
	\[ f^{\ast,\hyp} \colon \HSh(Y;\cE) \leftrightarrows \HSh(X;\cE) \colon f_\ast \ . \]
	Inspection reveals that $f^{\ast,\hyp}$ can be written as $(-)^{\hyp} \circ f\inv$.
\end{notation}

\begin{notation}
	For a topological space $X$, we denote by $\Gamma_X \colon X \to \ast$ the unique map.
\end{notation}

\begin{defin}
	Let $X$ be a topological space and let $\cE \in \PrL$.
	We say that a hypersheaf  $F \in \HSh(X;\cE)$ is \emph{hyperconstant} if $ F $ lies in the image of
	\begin{equation*}
		\Gamma_X^{\ast,\hyp} \colon \cE \to \HSh(X;\cE) \rlap{\ .}
	\end{equation*}
	We say that $ F $ is \emph{locally hyperconstant} if $ F $ is hyperconstant on an open cover of $X$.
	We denote by $\LChyp(X;\cE)$ the full subcategory of $\HSh(X;\cE)$ spanned by locally hyperconstant hypersheaves.
\end{defin}

\begin{defin}[{\cite{Lejay_Constructible}}]
	Let $(X,P)$ be a  stratified space and let $\cE\in \PrL$.
	We say that $F \in \HSh(X;\cE)$ is \emph{hyperconstructible} if for every $p \in P$, the hypersheaf $i_p^{\ast,\hyp}(F)$ is locally hyperconstant on $X_p$.
	We denote by $\ConsPhyp(X;\cE)$ the full subcategory of  $\HSh(X;\cE)$ spanned by hyperconstructible hypersheaves on $(X,P)$.
\end{defin}

\begin{rem}
      When $\cE=\Spc$, we write $\ConsPhyp(X)$ instead of $\ConsPhyp(X;\Spc)$.
\end{rem}

\begin{rem}\label{rem:pullback_functoriality}
	Given a map of stratified spaces $ f \colon (X,P) \to (Y,Q) $ and $\cE\in \PrL$, the hypersheaf pullback $f^{\ast,\hyp} : \HSh(Y;\cE) \to \HSh(X;\cE)$ restricts to a functor
	\begin{equation}\label{eq:pullback_functoriality}
		f^{\ast,\hyp} \colon \ConsQhyp(Y;\cE)\to \ConsPhyp(X;\cE) \ .
	\end{equation}
\end{rem}

\begin{warning}
	Let $k$ be a field and let $\Mod_k^\heartsuit$ be the abelian category of $k$-vector spaces.
	Then inside $\ConsPhyp(X;\Mod_k^\heartsuit)$ we allow constructible sheaves with infinite dimensional stalks.
	While this goes against the typical convention in algebraic geometry, it is crucial for us to allow such objects, for instance in the proof of \cref{toen_vaquié_comparaison_compact}. See also \cite[Warning 6.5.2]{Porta_Teyssier_Exodromy}.
\end{warning}

We conclude with a remark on the functoriality in the coefficient $\infty$-category $\cE$.

\begin{recollection}\label{recollection:functoriality_in_the_coefficients}
	Let $X$ be a topological space.
	Let $ L \colon \cE \to \cD $ be a left adjoint functor between presentable $ \infty $-categories with right adjoint $R \colon \cD \to \cE$.
	We denote by 
	\begin{equation*}
		L^{\hyp} \coloneqq \id \otimes L \colon \Shhyp(X;\cE) \to \Shhyp(X;\cD)
	\end{equation*}
	the induced left adjoint.
	Concretely, $L^{\hyp}$ is obtained via the composition of
	\begin{equation*}
		L \circ - \colon \PSh(X;\cE)  \to \PSh(X;\cD)
	\end{equation*}
	with the hypersheafification functor $(-)^{\hyp}$.
	Via the identification 
	\begin{equation*}
		\HSh(X;\cE) \simeq \FunR(\HSh(X)\op,\cE)
	\end{equation*}
	supplied by \cite[Proposition 4.8.1.17]{Lurie_Higher_algebra}, the right adjoint $R^{\hyp}$ of $L^{\hyp}$ is given by  composing with $R$.
\end{recollection}

\begin{rem}\label{hyp_coeff_change_inverse_image}
	The formation of $ L^{\hyp} $ commutes with hypersheaf pullback. 
\end{rem}

\begin{eg}\label{projection_formula}
	Let $A \in \dCRing$ and let $M \in \Mod_A$.
	In the setting of \cref{recollection:functoriality_in_the_coefficients}, we set $L \coloneqq M \otimes_A (-)$ and
	\[ M\otimes_A^{\hyp} ( - ) \coloneqq (M\otimes_A (-))^{\hyp} \colon \Shhyp(X;\Mod_A)\to \Shhyp(X;\Mod_A) \ . \]
	Similarly, given a morphism $A \to B$ in $\dCRing$, the adjunction
	\[ B \otimes_A (-) \colon \Mod_A \leftrightarrows \Mod_B \colon \Res_{B/A}(-) \ , \]
	given by extension and restriction of scalars induces an adjunction
	\[ B\otimes_A^{\hyp} ( - ) \colon \Shhyp(X;\Mod_A) \leftrightarrows \Shhyp(X;\Mod_B) \colon \Res_{B/A}^{\hyp} \ . \]
	The projection formula for modules immediately implies that for every $M\in Mod_B$ and every $F\in \Shhyp(X;\Mod_A)$,  the canonical morphism
	\begin{equation}\label{eq:projection_formula}
		\Res_{B/A}(M)\otimes_A^{\hyp}  F  \to \Res^{\hyp}_{B/A}(M\otimes_B^{\hyp}  (B\otimes_A^{\hyp}  F))
	\end{equation}
	is an equivalence.
\end{eg}


\subsection{Atomic generation and exodromic stratified spaces}\label{subsec:strat_spaces}

Following ideas of Clausen--Jansen \cite{Clausen_Jansen}, in \cite{Beyond_conicality} we introduced the special class of \emph{exodromic} stratified spaces that we recall below, together with their main properties and examples.

\begin{definition}[Atomic generation]
	Let $\cC \in \PrL$.
	\begin{enumerate}\itemsep=0.2cm
		\item An object $c \in \cC$ is \emph{atomic} (or \emph{absolutely compact}) if
		\[ \Map_\cC(c,-) \colon \cC \to \Spc \]
		commutes with all colimits.
		We write $\cC^{\at}$ for the full subcategory of $\cC$ spanned by atomic objects.
		
		\item We say that $\cC$ is \emph{atomically generated} if the unique colimit-preserving extension $\PSh(\cC^{\at}) \to \cC$ of the inclusion $ \cC^{\at} \hookrightarrow \cC $  is an equivalence.
	\end{enumerate}
\end{definition}

\begin{definition}[{Exodromic spaces, cf. \cite[Definition 3.5]{Beyond_conicality}}]\label{def:exodromic_stratified_space}
	A stratified space $ s \colon X\to P $ is \textit{exodromic} if the following conditions are satisfied:
	\begin{enumerate}\itemsep=0.2cm
		\item The $\infty$-category $ \ConsPhyp(X) $ is atomically generated.
		
		\item The subcategory $\ConsPhyp(X) \subset \Shhyp(X) $ is closed under both limits and colimits.
		
		\item The functor $s^{\ast,\hyp} \colon \Fun(P,\Spc) \simeq \Shhyp(P) \to \ConsPhyp(X)$ preserves limits.
	\end{enumerate}
	In this case, we define the \emph{$\infty$-category of exit-paths} of $(X,P)$ as
	\[ \Piinfty(X,P) \coloneqq (\ConsPhyp(X)^{\at})\op \ . \]
\end{definition}

Condition (2) should be thought as a \emph{categorical} regularity condition on the stratification, that weakens the conicality condition (see \cite[\S0.2]{Beyond_conicality} or \cite[\S III.2.5]{Porta_HDR} for additional comments on this topic).
The main source of examples is:

\begin{eg}[Conical stratifications]\label{eg:conical_stratifications_are_exodromic}
	Let $(X,P)$ be a conically stratified space with locally weakly contractible strata.
	Then $(X,P)$ is exodromic and moreover $\Pi_\infty(X,P)$ is computed by the simplicial model $\Sing_P(X)$ of \cite[Definition A.6.2]{Lurie_Higher_algebra}, see \cite[Theorem 5.17 \& Remark 5.18]{Porta_Teyssier_Exodromy}.
	In particular, if $(V,S)$ is a locally finite simplicial complex, then with the notation of \cref{eg:simplicial_complexes} $(\Delta^{(V,S)},S)$ is exodromic.
	Moreover, using \cite[Theorem A.6.10]{Lurie_Higher_algebra} there is a natural equivalence
	\begin{equation}\label{eq:Pi_infty_simplicial_complex}
		\Piinfty(\Delta^{(V,S)},S) \simeq S \ .
	\end{equation}
\end{eg}

Before stating further examples, let us recall a few important general features of the exodromy equivalence. 

\begin{theorem}[{Exodromy with coefficients, \cite[Corollary 4.1.15]{Beyond_conicality}}]\label{thm:exodromy_with_coefficients}
	Let $(X,P)$ be an exodromic stratified space.
	For every $\cE\in \PrL$, there is a canonical comparison functor
	\[ \Fun(\Piinfty(X,P), \cE) \to \ConsPhyp(X;\cE) \ . \] 
	If $\cE$ compactly assembled or $P$ noetherian and $\cE$ stable, then this functor is an equivalence.
\end{theorem}

\begin{theorem}[{Functoriality, \cite[Theorem 5.1.7-(3)]{Beyond_conicality}}]\label{thm:exodromy_functoriality}
	Let $f \colon (X,P) \to (Y,Q)$ be a morphism between exodromic stratified spaces and let $\cE $ be a compactly assembled presetnable $ \infty $-category.
	Then the functor $f^{\ast,\hyp}$ of \eqref{eq:pullback_functoriality} admits a left adjoint
	\[ f_\sharp^{\cons} \colon \ConsPhyp(X) \to \ConsQhyp(Y) \ , \]
	which furthermore restricts to a functor $\Pi_\infty(f) \colon \Pi_\infty(X,P) \to \Piinfty(Y,Q)$.
	Moreover for every $\cE$ as in \cref{thm:exodromy_with_coefficients}, the diagram
	\begin{equation}\label{exodromic_morphism_square}
		\begin{tikzcd}
			\ConsQhyp(Y;\cE) \arrow[r,"\sim"{yshift=-0.25em}] \arrow[d, "f^*"'] & \Fun(\Piinfty(Y,Q), \cE) \arrow[d, "- \circ \Piinfty(f)"] \\
			\ConsPhyp(X;\cE) \arrow[r,"\sim"{yshift=-0.25em}] & \Fun(\Piinfty(X,P), \cE)
		\end{tikzcd} 
	\end{equation}
	commutes.
\end{theorem}

\begin{recollection}[Objects of $\Piinfty(X,P)$]\label{obs:objects_of_Exit_with_locally_weakly_contractible_strata}
	Let $ (X,P) $ be an exodromic stratified space with locally weakly contractible strata.
	There is a natural identification
	\begin{equation*}
		\Piinfty(X,P)^{\simeq} \simeq \coprod_{p \in P} \Piinfty(X_p) \ ,
	\end{equation*}
	where $ \Piinfty(X_p) $ denotes the underlying homotopy type of the stratum $ X_p $.
	See \cite[Observation 5.1.6]{Beyond_conicality}.
	Thus, every point $x \in X$ gives rise to an object
	\[ [x] \coloneqq x_\sharp^{\cons}(\ast) \in \ConsPhyp(X) \ , \]
	and every object of $\Piinfty(X,P)$ is of this form.
	Notice that, by design, $x_\sharp^{\cons}(\ast)$ corepresents the stalk functor
	\begin{equation*}
		x^\ast \colon \ConsPhyp(X) \to \Spc
	\end{equation*}
	On the other hand, different points of $X$ might give rise to the same object of $\Piinfty(X,P)$.
\end{recollection}

\begin{lem}\label{compact_generators_Cons}
    Let $ (X,P) $ be an exodromic stratified space with locally weakly contractible strata and let $\cE $ be a compactly generated presentable \category.
    Let $(e_{\alpha})_{\alpha \in A}$ be a system of compact generators for $\cE$.
    Then, the the objects $(x_{\sharp}^{\cons}(e_{\alpha}))_{x\in X, \alpha \in A}$ form a system of compact generators for \smash{$\ConsPhyp(X,\cE)$}.
\end{lem}

\begin{proof}
      Same proof as in \cite[Theorem 6.4.1]{Porta_Teyssier_Exodromy} using \eqref{exodromic_morphism_square} and \cite[5.3.2.11-(1)]{Lurie_Higher_algebra}.
\end{proof}

Exodromic stratifications allow to go beyond conical examples, mainly thanks to:

\begin{theorem}[{Stability under coarsening, \cite[Theorem 3.3.5]{Beyond_conicality}}]\label{thm:stability_coarsening}
	Let $ (X,R) $ be an exodromic stratified space, and let $ \phi \colon R\to P $ be a map of posets.
	Then:
	\begin{enumerate}\itemsep=0.2cm
		\item\label{thm:stability_under_coarsening.1} The stratified space $ (X,P) $ is exodromic.
		
		\item\label{thm:stability_under_coarsening.2} The natural functor $ \Piinfty(X,R)\to \Piinfty(X,P) $ induces an equivalence 
		\[ \Piinfty(X,R)[W_P\inv]\simeq \Piinfty(X,P) \ , \]
		where $ W_P \subset \Mor(\Piinfty(X,R)) $ is the collection of morphisms sent to equivalences by the composite $\Piinfty(X,R) \to R \to P $.
	\end{enumerate}
\end{theorem}

\begin{eg}[Subanalytic stratifications]\label{eg:subanalytic_stratifications}
	Let $(M,X,P)$ be a subanalytic stratified space, in the sense of \cref{def:subanalytic_stratified_space}.
	By \cite{Verdier1976}, the stratified space $(X,P)$ admits a refinement by a locally finite triangulation.
	Combining \cref{eg:conical_stratifications_are_exodromic} and \cref{thm:stability_coarsening}, it follows that $(X,P)$ is exodromic. See \cite[Theorem 5.3.9]{Beyond_conicality}.
\end{eg}

\begin{eg}[Algebraic stratifications]
	Let $(X,P)$ be an algebraic stratified space, in the sense of \cref{def:algebraic_stratified_space}.
	Then $(X,P)$ is exodromic, see \cite[Theorem 5.3.13]{Beyond_conicality}: this is essentially a consequence of the previous example together with the van Kampen theorem for exodromic spaces, see \cite[Theorem 5.1.7-(4)]{Beyond_conicality}.
\end{eg}


\subsection{Finiteness conditions}

Fundamental to the construction of moduli spaces are the following finiteness conditions on exodromic stratified spaces:

\begin{defin}\label{categorically_compact}
	For an exodromic stratified space $(X,P)$, we say that:
	\begin{enumerate}\itemsep=0.2cm
		\item $(X,P)$ is \emph{categorically compact} if $\Piinfty(X,P)$ is a compact object in $\Cat_\infty$ ;

		\item $(X,P)$ is \emph{locally categorically compact} if $X$ admits a fundamental system of open subsets $U$ such that $(U,P)$ is exodromic and categorically compact.
		\end{enumerate}
\end{defin}		

\begin{eg}[Simplicial complexes]\label{finite_geometric_realization_simplicial_complex}
	Let $(V,S)$ be a locally finite simplicial complex.
	Equation \eqref{eq:Pi_infty_simplicial_complex} shows that the exodromic stratified space $(\Delta^{(V,S)}, S)$ is locally categorically compact and that it is categorically compact if and only if the set $S$ is finite.
\end{eg}

The following are the two main finiteness results proven in \cite{Beyond_conicality}. They generalize the classical finiteness theorems of homotopy groups of Lefschetz--Whitehead, Łojasiewicz, and Hironaka.
They are the conceptual basis for the rest of this paper.

\begin{theorem}[Subanalytic finiteness]\label{thm:subanalytic_stratified_spaces_are_conically_refineable}
	Let $ (M,X,P) $ be a subanalytic stratified space.
	Then:
	\begin{enumerate}\itemsep=0.2cm
		\item\label{thm:subanalytic_stratified_spaces_are_conically_refineable.1} The stratified space $ (X,P) $ admits a refinement by a locally finite triangulation. 
		
		\item\label{thm:subanalytic_stratified_spaces_are_conically_refineable.2} The stratified space $ (X,P) $ is  exodromic with locally weakly contractible strata.
		
		\item\label{thm:subanalytic_stratified_spaces_are_conically_refineable.3} The stratified space $ (X,P) $ is locally categorically compact. 
		
		\item\label{thm:subanalytic_stratified_spaces_are_conically_refineable.4} If $ X $ is compact, then $ (X,P) $ admits a refinement by a \emph{finite} triangulation.
		Hence $ (X,P) $ is categorically compact.
	\end{enumerate}
\end{theorem}

\begin{theorem}[Algebraic finiteness]\label{thm:algebraic_stratified_spaces_are_conically_refineable_and_categorically_finite}
	Let $ (X,P) $ be an algebraic stratified space.
	Then:
	\begin{enumerate}\itemsep=0.2cm
		\item\label{thm:algebraic_stratified_spaces_are_conically_refineable_and_categorically_finite.1} If $ X $ is affine, $ (X,P) $ admits a categorically compact locally categorically compact conical refinement $ (X,R) $ with locally weakly contractible strata and with $ R $ finite.
		
		\item\label{thm:algebraic_stratified_spaces_are_conically_refineable_and_categorically_finite.3} The stratified space $ (X,P) $ is exodromic  with locally weakly contractible strata.
		
		\item\label{thm:algebraic_stratified_spaces_are_conically_refineable_and_categorically_finite.4} The stratified space $ (X,P) $ is categorically compact and locally categorically compact.
	\end{enumerate}
\end{theorem}

\begin{rem}
	In the conical setting, stronger finiteness results have been obtained by Volpe in \cite[Theorem 1.5 \& Corollary 1.6]{Volpe_Finiteness}.
\end{rem}

The proofs of \Cref{thm:subanalytic_stratified_spaces_are_conically_refineable,thm:algebraic_stratified_spaces_are_conically_refineable_and_categorically_finite} rely on the following more general stability properties:

\begin{prop}[{\cite[Theorem 3.6.2]{Beyond_conicality}}]\label{lem:pulling_back_to_a_locally_closed_subposet_preserves_categorical_finiteness_and_compactness}
	Let $(X,P)$ be an exodromic stratified space and $S \subset P$ a locally closed subposet.
	If $(X,P)$ is categorically  compact, then so is $(X_S, S)$.
\end{prop}

\begin{proposition}[{\cite[Theorem 3.6.4]{Beyond_conicality}}]\label{prop:categorical_finiteness_and_compactness_are_stable_under_coarsening}
	Let $(X,R)$ be an exodromic stratified space and let $R \to P$ be a map of posets.
	If $(X,R)$ is categorically compact, then so is $(X,P)$.
\end{proposition}


\subsection{\texorpdfstring{$!$}{!}-pullback functoriality}

We conclude this section by recalling $ ! $-pullback functoriality of (constructible) sheaves its main properties, mostly following \cite{Porta_Teyssier_Exodromy}.
Let us highlight that a number of important results in the constructible setting require local categorical compactness.

\begin{recollection}\label{locally_closed_lower_shriek}
	Let $X$ be a topological space, let $i \colon Y \hookrightarrow X$ be a closed immersion and let $j \colon U \hookrightarrow X$ be the open complement.
	By \cite[Corollary 2.18 \& Proposition 2.26]{Haine_Nonabelian_basechange} for every presentable stable $\infty$-category $\cE$, the fully faithful functors 
	\[ 
	i_{\ast} \colon \HSh(Y;\cE) \hookrightarrow \HSh(X;\cE)   \hookleftarrow  \HSh(U;\cE)\colon   j_{\ast}  
	\]
	exhibit $\HSh(X;\cE)$ as a recollement of $\HSh(Y;\cE)$  and $\HSh(U;\cE)$ in the sense of \cite[A.8.1]{Lurie_Higher_algebra}.
	In particular, $i_{\ast} \colon  \HSh(Y;\cE) \hookrightarrow \HSh(X;\cE)$ admits a right adjoint 
	\[
	i^{!,\hyp} \colon \HSh(X;\cE) \to \HSh(Y;\cE) \ ,
	\]
	whereas $j^\ast$ admits a left adjoint $j_!^{\hyp}$.
	For a locally closed immersion $i \colon Y \hookrightarrow X$ between topological spaces presented as the composition of a closed immersion $\iota \colon Y\hookrightarrow U$ followed by an open immersion $j \colon U\hookrightarrow X$, we set 
	\[
	i^{!,\hyp} \coloneqq \iota^{!,\hyp} \circ j^{\ast, \hyp} \colon \HSh(X;\cE) \to\HSh(Y;\cE)  \  .
	\]
	This does not depend on the choice of the presentation of $i$.
	Observe also that $i^{!,\hyp}$  admits a left adjoint given by 
	\[
		i_!^{\hyp} \coloneqq j_!^{\hyp} \circ \iota_\ast \colon \HSh(Y;\cE) \to \HSh(X;\cE)  \  .
	\]
\end{recollection}
	
\begin{eg}\label{i_S_shrieck}
	Let $(X,P)$ be a stratified space. 
	Let $S \subset P$ be a locally closed subset.
	Define
	\[ \geq S \coloneqq\{p\in P | \textup{ there exists } s\in S \text{ such that } p\geq s\} \]
	The set $\geq S$ is open in $P$ and $S$ is closed in $\geq S$.
	Thus, the inclusion $i_S \colon X_S \to X$ factors as
	\[ \begin{tikzcd}[column sep = small]
		X_S \arrow{r}{\iota_S} & X_{\geqslant S} \arrow{r}{i_{\geqslant S}} & X \ ,
	\end{tikzcd} \]
	where $\iota_S$ is a closed immersion and $i_{\geqslant S}$ is an open immersion.
\end{eg}
	
\begin{lem}[{\cite[Lemma 2.5.12]{Porta_Teyssier_Exodromy}}]\label{right_adjoint_and_!_closed_immersion}
	Let $i \colon Y \hookrightarrow X$ be a locally closed immersion between topological spaces presented as the composition of a closed immersion $\iota \colon Y\hookrightarrow U$ followed by an open immersion $j \colon U\hookrightarrow X$.
	Let $ L \colon \cE \rightleftarrows \cD \noloc R $ be an adjunction between presentable stable $ \infty $-categories.
	Then, the squares
	\[ \begin{tikzcd}
			\HSh(Y;\cE)  \arrow{r}{L^{\hyp}} \arrow[d, "i_!^{\hyp}"'] & \HSh(Y;\cD)\arrow{d}{i_!^{\hyp}} \\
			\HSh(X;\cE)\arrow[r, "L^{\hyp}"'] & \HSh(X;\cD)  
	\end{tikzcd} 
	\quad \text{and} \quad
	\begin{tikzcd}
		\HSh(Y;\cE)  &\HSh(Y;\cD)\arrow{l}[swap]{R^{\hyp}}\\
		\HSh(X;\cE)\arrow{u}{i^{!,\hyp}}&  \HSh(X;\cD)\arrow{l}{R^{\hyp}} \arrow[u, "i^{!,\hyp}"']
	\end{tikzcd} \]
	commute.
\end{lem}
	
In general, the $ ! $-pullbacks functors do not preserve $P$-hyperconstructibility: this is indeed a special feature of the conical setting.

\begin{notation}[{\cite[\S6.5]{Porta_Teyssier_Exodromy}}]
	Let $(X,P)$ be a conically stratified space and let $\cE$ be a compactly generated presentable stable $\infty$-category.
	We denote by $\Cons_{P,\omega}^{\hyp}(X;\cE)$ the full subcategory of $\ConsPhyp(X;\cE)$ spanned by $P$-hyperconstructible hypersheaves with compact stalks.
	Under the exodromy equivalence, they correspond to $\Fun(\Pi_\infty(X,P), \cE^\omega)$.
\end{notation}

\begin{prop}\label{prop:uppershriek_and_constructibility}
	Let $(X,P)$ be a conically stratified space with locally weakly contractible strata and let $S \subseteq P$ be a locally closed subset.
	Let $\cE$ be a compactly generated presentable stable $\infty$-category.
	\begin{enumerate}\itemsep=0.2cm
		\item\label{prop:uppershriek_and_constructibility:preservation} \cite[Proposition 6.8.3]{Porta_Teyssier_Exodromy} The functor $i_S^{!,\hyp} \colon \HSh(X;\cE)  \to  \HSh(X_S;\cE)$ restricts to a functor
		\[ i_S^{!,\hyp} \colon \ConsPhyp(X;\cE) \to \ConsShyp(X_S;\cE)  \ . \]
	\end{enumerate}
	Assume in addition that $(X,P)$ is locally categorically compact and that $P$ is finite. Then:
	\begin{enumerate}\itemsep=0.2cm\setcounter{enumi}{1}
		\item\label{prop:uppershriek_and_constructibility:uppershriek_and_compact_stalks} \cite[Proposition 6.8.4]{Porta_Teyssier_Exodromy} The functor $i_S^{!,\hyp}$ restricts to a functor
		\[ i_S^{!,\hyp} \colon \Cons_{P,\omega}^{\hyp}(X;\cE) \to \Cons_{S,\omega}^{\hyp}(X_S;\cE)\ . \]
		
		\item\label{prop:uppershriek_and_constructibility:uppershriek_change_of_coefficients} \cite[Proposition 6.9.6]{Porta_Teyssier_Exodromy} For every morphism $L \colon \cE \to \cD$ of presentable stable $\infty$-categories, the commutative square
		\[ \begin{tikzcd}
			\HSh(X_S;\cE) \arrow{r}{L^{\hyp} } \arrow[d, "i_{S,!}^{\hyp}"'] & \HSh(X_S;\cD) \arrow{d}{i_{S,!}^{\hyp}} \\
			\HSh(X;\cE)  \arrow[r, "L^{\hyp}"'] & \HSh(X;\cE)  
		\end{tikzcd} \]
		is vertically right adjointable on $P$-hyperconstructible hypersheaves.
		That is, for every $F\in \ConsPhyp(X;\cE)$, the exchange transformation
		\[  L^{\hyp} \circ i_S^{!,\hyp}(F) \to i_S^{!,\hyp} \circ L^{\hyp}(F)  \]
		is an equivalence.
	\end{enumerate}
\end{prop}

\section{Perverse sheaves and flat families}\label{sec:perverse_t_structure}

In this section we collect, the basic facts about the perverse $t$-structure and about how it varies in families.
The main result is \cref{prop:perverse_flatness_reformulation}, which allows to prove that the perverse $t$-structure \emph{universally satisfies openness of flatness} in the sense of \cite[Definition II.2.54]{DPS_Stable_pairs}.

\begin{convention}
	Throughout the whole section we fix a derived commutative ring $A$ and consider the associated presentable stable $\infty$-category $\Mod_A$, equipped with its standard $t$-structure $\tau = ((\Mod_A)_{\geqslant 0}, (\Mod_A)_{\leqslant 0})$, see \cite[Proposition 7.1.1.13]{Lurie_Higher_algebra}.
	In particular $\tau$ is accessible, left and right complete and compatible with filtered colimits.
\end{convention}


\subsection{Hypersheaves and $t$-structures}\label{t_structure_on_hypersheaves}

We start by recalling the basics of the induced $ t $-structure on $ \Mod_A $-valued hypersheaves.

\begin{recollection}\label{recollection:standard_t_structure}
	Let $X$ be a topological space.
	Recall from \cite[Proposition 2.1.1.1]{Lurie_SAG} that $\HSh(X;\Mod_A)$ is equipped with a $t$-structure
	\[ \tau^{\hyp}_X = \big( \HSh(X;\Mod_A)_{\geqslant 0}, \HSh(X;\Mod_A)_{\leqslant 0} \big) \]
	whose coconnective part is
	\[ \HSh(X;\Mod_A)_{\leq 0} \coloneqq \HSh(X;(\Mod_A)_{\leqslant 0}) \ . \]
	It follows from \emph{loc.\ cit.} that \smash{$\tau_X^{\hyp}$} is accessible, left complete, and compatible with filtered colimits.
	We refer to \smash{$\tau_X^{\hyp}$} as the \emph{standard t-structure} on \smash{$\HSh(X;\Mod_A)$}.
\end{recollection}

\begin{recollection}\label{recollection:t_exactness}
	Let $f \colon X \to Y$ be a map of topological spaces.
	Then \cite[Remark 1.3.2.8]{Lurie_SAG} shows that the functor
	\[ f^{\ast,\hyp} \colon \HSh(Y;\Mod_A) \to \HSh(X;\Mod_A) \]
	is $t$-exact with respect to the standard $t$-structures $\tau^{\hyp}_X$ and $\tau^{\hyp}_Y$.
	Consequently, its right adjoint 
	\begin{equation*}
		f_\ast \colon \HSh(X;\Mod_A) \to \HSh(Y;\Mod_A)
	\end{equation*}
	is left (but typically not right) $t$-exact.
\end{recollection}

We make repeated use of the following elementary observation:

\begin{lem}\label{lem:conservativity_trick}
	Let $f \colon \cE \to \cE'$ be a functor in $\PrL$.
	Assume that $\cE$ and $\cE'$ are stable and equipped with $t$-structures $\tau = (\cE_{\geqslant 0}, \cE_{\leqslant 0})$ and $\tau' = (\cE'_{\geqslant 0}, \cE'_{\leqslant 0})$.
	If $f$ is $t$-exact and conservative, then an object $E \in \cE$ is (co)connective if and only if $f(E)$ is (co)connective.
\end{lem}

\begin{proof}
	Write $\tau_{\leqslant n}$ and $\tau_{\leqslant n}'$ for the truncation functors in $\cE$ and $\cE'$, respectively.
	Let $E \in \cE$ be an object and assume that $f(E) \in \cE'_{\geqslant n}$.
	Then
	\[ 0 \simeq \tau'_{\leqslant n-1} f(E) \simeq f (\tau_{\leqslant n-1}(E) ) \ , \]
	where the second equivalence follows from $t$-exactness of $f$.
	Since $f$ is conservative, $\tau_{\leqslant n-1}(E) \simeq 0$, i.e., $E \in \cE_{\geqslant n}$.
	The symmetric argument (using the other truncation functors) shows that if $f(E) \in \cE'_{\leqslant n}$, then $E \in \cE_{\leqslant n}$ as well.
\end{proof}

We collect below some well-known properties of the standard $t$-structure.

\begin{cor}\label{cor:t-structure_on_stalks}
	Let $X$ be a topological space.
	Then $F \in \HSh(X;\Mod_A)$ is (co)connective if and only if its stalks are (co)connective.
\end{cor}

\begin{proof}
	\cref{recollection:t_exactness} and \cite[Recollection 1.7]{HPT} imply that the stalk functors
	\[ \{x^{\ast,\hyp} \colon \HSh(X;\Mod_A) \to \Mod_A\}_{x \in X} \]
	are $t$-exact and jointly conservative.
	The conclusion then follows from \cref{lem:conservativity_trick}.
\end{proof}

\begin{cor}\label{locality_t_structure}
	Let $X$ be a topological space and let $\{U_i\}_{i \in I}$ be an open cover of $X$.
	Then a hypersheaf $F \in \HSh(X;\Mod_A)$ is (co)connective if and only if for each $i \in I$, the restriction $F |_{U_i}$ is (co)connective.
\end{cor}

\begin{proof}
	Immediate from \cref{cor:t-structure_on_stalks}.
\end{proof}

\begin{rem}\label{Res_exact}
    If  $A \to B$ is a morphism in $\dCRing$, the functor 
    \[ \Res_{B/A}^{\hyp} \colon \HSh(X; \Mod_B) \to \HSh(X; \Mod_A) \]
    introduced in \cref{projection_formula} is conservative and $t$-exact with respect to the standard $t$-structures from \cref{recollection:standard_t_structure}.
    Thus, \cref{lem:conservativity_trick} shows that $F \in \HSh(X;\Mod_B)$ is (co)connective if and only if so is $\Res_{B/A}^{\hyp}(F) \in \HSh(X;\Mod_A)$.
\end{rem}


\begin{lem}\label{more_exactness}
	Let $X$ be a topological space. 
	The following hold:
	\begin{enumerate}\itemsep=0.2cm
		\item For every open immersion $j  \colon U \hookrightarrow X$, the functor
		\[ j_! \colon \HSh(U;\Mod_A) \hookrightarrow \HSh(X;\Mod_A) \]
		is right $\tau^{\hyp}_X$-exact.

		\item For every locally closed immersion $i \colon Y\to X$, the functor
		\[ i^{!,\hyp} \colon \HSh(X;\Mod_A) \to \HSh(Y;\Mod_A) \]
		of \cref{locally_closed_lower_shriek} is left $\tau^{\hyp}_X$-exact.
		
		\item For every closed immersion $i \colon Y \hookrightarrow X$, the functor
		\[ i_\ast \colon \HSh(Y;\Mod_A) \hookrightarrow \HSh(X;\Mod_A) \]
		is $\tau^{\hyp}_X$-exact.
	\end{enumerate}
\end{lem}

\begin{proof}
	Item (1) follows from \cref{recollection:t_exactness} as $j_!$ is left adjoint to $j^{\ast, \hyp}$.
	We now prove (2).
	Choose a presentation of $i$ as a closed immersion $\iota \colon Y \hookrightarrow U$ followed by an open immersion $j \colon U \hookrightarrow X$, so that $i^{!,\hyp} =\iota^{!,\hyp}\circ j^{\ast,\hyp}$.
	By \cref{recollection:t_exactness}, $ j^{\ast,\hyp} $ is $ t $-exact; hence can reduce to the case where $i$ is a closed immersion.
	Let $j \colon X \smallsetminus Y \hookrightarrow X$ be the inclusion of the open complement.
	For $F\in \HSh(X;\Mod_A)$, there is a fiber sequence 
	\[ i^{!,\hyp}(F) \to i^{\ast,\hyp}(F) \to i^{\ast,\hyp}j_\ast j^{\ast,\hyp}(F) \ . \]
	Thus, if $F\in \HSh(X;\Mod_A)_{ \leqslant0}$, \cref{recollection:t_exactness} shows that the middle and right terms are also coconnective. 
	Since $\HSh(X;\Mod_A)_{ \leqslant0}$ is stable under limits \cite[Corollary 1.2.1.6]{Lurie_Higher_algebra}, the conclusion follows.
	
	To prove (3), by \cref{recollection:t_exactness}, it remains to prove that $i_\ast$ is right $\tau^{\hyp}_X$-exact.
	This follows from item (2) as $i_\ast$ is left adjoint to $i^{!,\hyp}$.
\end{proof}

%


\begin{lem}\label{coconnectivity_and_strata}
	Let $(X,P)$ be a stratified space with $P$ noetherian.
	Then $F\in \HSh(X;\Mod_A)$ is coconnective if and only if for each $ p \in P $, the $ ! $-restriction $i_p^{!, \hyp}(F)$ is coconnective.
\end{lem}

\begin{proof}
	The direct implication follows from \cref{more_exactness}-(2).
	Let us show the converse.
	Since the open subsets $\{P_{\geq p}\}_{p\in P}$ cover $P$, by \cref{locality_t_structure} it suffices to show that 
	\begin{equation}\label{eq_coconnectivity_and_strata}
		i_{P_{\geq p}}^{\ast, \hyp}(F) \text{ is coconnective }
	\end{equation}
    for each $p\in P$.
	
	We argue by contradiction and assume that \eqref{eq_coconnectivity_and_strata} fails for some $p\in P$.
	Since $P$ is noetherian, we can choose $p$ to be maximal. 
	Hence,  for every $q\in P_{>p}$,  the hypersheaf $i_{P_{\geq q}}^{\ast, \hyp}(F)$  is coconnective.
	Since the open subsets $\{P_{\geq q}\}_{q>p}$ cover $P_{>p}$, \cref{locality_t_structure} implies that \smash{$i_{P_{> p}}^{\ast, \hyp}(F)$} is coconnective.
	Let $i \colon X_p \hookrightarrow X_{P_{\geq p}}$ and $j \colon X_{P_{>p}} \hookrightarrow X_{P_{\geq p}}$ be the inclusions and consider the fibre sequence 
	\[ i_{\ast}i_p^{!, \hyp}(F) \to  i_{P_{\geq  p}}^{\ast \hyp}(F)   \to  	j_{\ast} i_{P_{>  p}}^{\ast, \hyp}(F) \ . \]
	Then, by assumption combined with \cref{recollection:t_exactness}, the right and left terms of the above sequence are coconnective.
	Hence so is the middle term, a contradiction.
\end{proof}


\subsection{Perverse $t$-structures}

In the following definition, we make use of the $t$-structure on hypersheaves from  \cref{t_structure_on_hypersheaves}.

\begin{defin}\label{def_t_structure_perv}
	Let $(X,P)$ be a stratified space with $P$ finite, let $A \in \dCRing$ and $\pfrak \colon P \to \mathbb Z$ be a function.
	We define
	\[ {}^{\pfrak} \HSh(X;\Mod_A)_{ \geqslant0}  \coloneqq   \big\{F \in \HSh(X;\Mod_A)\mid \forall p \in P \ , \: i_p^{\ast,\hyp}(F)\in \HSh(X_p;\Mod_A)_{\geq \pfrak(p)}\ \big\} \]
	and
	\[ {}^{\pfrak} \HSh(X;\Mod_A)_{\leqslant 0}  \coloneqq \big\{F \in \HSh(X;\Mod_A)\mid \forall p \in P \ , \:  i_p^{!,\hyp}(F)\in \HSh(X_p;\Mod_A)_{\leq \pfrak(p)}\ \big\} \]
\end{defin}

\begin{lem}
    In the setting of \cref{def_t_structure_perv}, the subcategories 
	\[ ({}^{\pfrak} \HSh(X;\Mod_A)_{ \geqslant 0},  {}^{\pfrak} \HSh(X;\Mod_A)_{\leqslant 0}) \]
	define a $t$-structure on $\HSh(X;\Mod_A)$.
\end{lem}

\begin{proof}
     The proof is the same as \cite[Corollary 2.1.4]{BBD}.
     Namely we argue by recursion on the cardinality of $P$.
     If $P=\{\ast\}$, \cref{def_t_structure_perv} produces  a shift of the standard $t$-structure on $\HSh(X;\Mod_A)$.
     Assume that $P$ has at least two elements and pick $p\in P$ maximal.
     Put $Q\coloneqq P\setminus p$ and let $\mathfrak q \coloneqq \pfrak|_Q \colon Q\to \bZ$.
     Then, the conclusion follows from the observation that
	\[ ({}^{\pfrak} \HSh(X;\Mod_A)_{ \geqslant 0}  , {}^{\pfrak} \HSh(X;\Mod_A)_{ \leqslant 0} ) \]     
	is the $t$-structure obtained by gluing  the $t$-structure
	\[ ({}^{\mathfrak q} \HSh(X_Q;\Mod_A)_{ \geqslant 0}  , {}^{\mathfrak q} \HSh(X_Q;\Mod_A)_{ \leqslant 0} ) \]
	with the shifted standard $t$-structure
	\[ (\HSh(X_p;\Mod_A)_{ \geqslant \pfrak(p)},\HSh(X_p;\Mod_A)_{ \leqslant \pfrak(p)} ) \ . \qedhere \]
\end{proof}

\begin{notation}\label{notation:perverse_t_structure}
    We refer to the $t$-structure from \cref{def_t_structure_perv} as the \textit{$\pfrak$-perverse $t$-structure} and denote it by \smash{$\ptauhypX$}.
    We write 
	\[ {}^{\pfrak} \mathrm{Perv}^{\hyp}(X;\Mod_A) \subset \HSh(X;\Mod_A) \]
	for its heart.
\end{notation}

\begin{lem}\label{Perv_scalar_restriction_prep}
    Let $(X,P)$ be a  stratified space with $P$ finite and let $\pfrak \colon P \to \mathbb Z$ be a function.
    Let $A \to B$ be a morphism in $\dCRing$.
    Then $F\in \HSh(X;\Mod_B)$ is $\ptauhypX$-(co)connective if and only if $\Res^{\hyp}_{B/A}(F)\in \HSh(X;\Mod_A)$ is $\ptauhypX$-(co)connective.
\end{lem}
\begin{proof}
    For connectiveness, the claim follows from Remarks \ref{Res_exact} and \ref{hyp_coeff_change_inverse_image}.
    For the coconnectiveness, the claim follows from \cref{Res_exact} and \cref{right_adjoint_and_!_closed_immersion}.
\end{proof}

\begin{lem}\label{Perv_scalar_restriction}
	Let $(X,P)$ be a  stratified space with $P$ finite and let $\pfrak \colon P \to \mathbb Z$ be a function.
	Let $A \to B$ be a morphism in $\dCRing$.
	Then, there is a pullback square in $\Cat_{\infty}$
	\[ \begin{tikzcd}[row sep=2.5em, column sep=3.5em]
		{}^{\pfrak}\mathrm{Perv}^{\hyp}_P(X;\Mod_B) \arrow{r}{\Res^{\hyp}_{B/A}} \arrow[hook]{d} &{}^{\pfrak}\mathrm{Perv}^{\hyp}_P(X;\Mod_A)\arrow[hook]{d}\\
		\HSh(X;\Mod_B)\arrow[r, "\Res^{\hyp}_{B/A}"'] & \HSh(X;\Mod_A) \ .
	\end{tikzcd} \]
\end{lem}
\begin{proof}
	Immediate from \cref{Perv_scalar_restriction_prep}.
\end{proof}

\begin{lem}\label{Perv_and_refinement}
	Let $(X,R)$ be a stratified space with $R$ finite and let $\phi \colon R \to P$ be a map of finite posets. 
    Let $\pfrak \colon P \to \mathbb Z$ be a function and let $\mathfrak r \coloneqq \pfrak \phi \colon R \to \mathbb Z$ be the composite.
	Let $A \in \dCRing$.
	Then, for every $F \in \HSh(X;\Mod_A)$ the following hold :
	\begin{enumerate}\itemsep=0.2cm
		\item $F$ is $\ptauhypX$-connective if and only if it is $\rtauhypX$-connective.
		\item $F$ is $\ptauhypX$-coconnective if and only if it is $\rtauhypX$-coconnective.
	\end{enumerate}
\end{lem}
\begin{proof}
	Apply \cref{cor:t-structure_on_stalks} and \cref{coconnectivity_and_strata}.
\end{proof}

The following lemma allows us to verify openness of the perverse $t$-structure.
    
\begin{lem}\label{lem:checking_perverse_on_stalks}
	Let $(X,P)$ be a stratified space with $P$ finite and let $\pfrak \colon P \to \bZ$ be a function.
	Then:
	\begin{enumerate}\itemsep=0.2cm
		\item\label{perverse_connective_stalks} A hypersheaf $F\in \HSh(X;\Mod_A)$ is \smash{$\ptauhypX$}-connective if and only if for each $p\in P$, the stalks of $i_p^{\ast,\hyp}(F)$ are $\pfrak(p)$-connective.
		
		\item\label{perverse_coconnective_stalks} A hypersheaf $F\in \HSh(X;\Mod_A)$ is \smash{$\ptauhypX$}-coconnective if and only if for each $p\in P$, the stalks of $i_p^{!,\hyp}(F)$ are $\pfrak(p)$-coconnective.
	\end{enumerate}
\end{lem}

\begin{proof}
    By definition $F$ is \smash{$\ptauhypX$}-(co)connective if and only if for each $p \in P$ one has $i_p^{\ast,\hyp}(F) \in \HSh(X_p;\Mod_A)_{\geq \pfrak(p)}$ (resp., $i_p^{!,\hyp}(F) \in \HSh(X_p;\Mod_A)_{\geq \pfrak(p)}$). Thus, the conclusion follows from \cref{cor:t-structure_on_stalks}.
\end{proof}

%


\subsection{Flatness}

From this point on, we make heavy use of the notations from \cref{projection_formula}.
The following is the translation of \cite[Definition II.2.44]{DPS_Stable_pairs} in our special setting:

\begin{defin}\label{def:perverse_flatness}
	Let $(X,P)$ be a  stratified space with P finite and let $\pfrak \colon P \to \mathbb Z$ be a function.
	For $A \in \dCRing$, we say that an $A$-linear hypersheaf $F \in \HSh(X;\Mod_A)$ is \emph{\smash{$\ptauhypX$}-flat relative to $A$} if for each $M \in \Mod_A^\heartsuit$, we have
	\[ M \otimes_A^{\hyp} F \in {}^{\pfrak} \mathrm{Perv}^{\hyp}(X;\Mod_A) \ . \]
\end{defin}

\begin{warning}
	Recall that $A$ lies in $\Mod_A^\heartsuit$ if and only if the natural map $A\to \pi_0(A)$ is an equivalence, that is when $A$ is \textit{underived}.	
	In that case, a \smash{$\ptauhypX$}-flat object automatically lies in ${}^{\pfrak} \mathrm{Perv}^{\hyp}(X;\Mod_A)$.
	However, unless $A$ is a field, objects in ${}^{\pfrak} \mathrm{Perv}^{\hyp}(X;\Mod_A)$ are not automatically $A$-flat.
\end{warning}

\begin{lem}[{\cite[Lemma II.2.45]{DPS_Stable_pairs}}]\label{tau_flatness_stable_under_pullback}
     Let $(X,P)$ be a  stratified space with $P$ finite and let $\pfrak \colon P \to \mathbb Z$ be a function.
     Let $A \to B$ be a morphism in $\dCRing$ and let $F \in \HSh(X;\Mod_A)$.
     If $F$ is \smash{$\ptauhypX$}-flat relative to $A$, then $B\otimes_A^{\hyp} F$ is \smash{$\ptauhypX$}-flat relative to $B$.
\end{lem}
\begin{proof}
     Immediate from \cref{Perv_scalar_restriction_prep} and the projection formula from \cref{projection_formula}.
\end{proof}

\cref{prop:perverse_flatness_reformulation} provides a stalkwise reformulation of $\ptau$-flatness.

\begin{lem}\label{shriek_and_tensor_product}
	Let $(X,P)$ be a conically stratified space with locally weakly contractible strata and let $S\subset P$ be a locally closed subset. 
	Assume that $(X,P)$ is locally categorically compact and  that $P$ is finite.
	Then for every $A \in \dCRing$, every $M \in \Mod_A$, and every $F \in \Cons_P(X;\Mod_A)$, the natural transformation 		
		\[  M \otimes_A^{\hyp} i_S^{!,\hyp}(F) \longrightarrow i_S^{!,\hyp}( M \otimes_A^{\hyp} F )  \] is an equivalence.
\end{lem}

\begin{proof}
	Apply \cref{prop:uppershriek_and_constructibility}-\eqref{prop:uppershriek_and_constructibility:uppershriek_change_of_coefficients} to $L=M \otimes_A (-) : \Mod_A \to \Mod_A$.
\end{proof}

\personal{
     Ok with the fact that connectivness coincides with positive $ \Tor $-amplitude for eventually connective objects.  
     To keep the symmetry between the connective and coconnective stalkwise criteria for flatness below,  let me stick to the tor amplitude viewpoint,  as this is the only one relevant for moduli.
\begin{lem}\label{tor_spectral}
    Let $A \in \dCRing$.
    Let $F\in \Mod_A$ and assume that $\pi_i(F) \simeq 0$ for $i \ll 0$.
    For every $k\in \Z$,  the following are equivalent : 
\begin{enumerate}\itemsep=0.2cm
			\item $F$ is $k$-connective, that is it belongs to $(\Mod_A)_{\geqslant k}$;
			\item for every $M \in \Mod_A^\heartsuit$, the module $M \otimes_A F$ is $k$-connective, that is $F$ has $ \Tor $-amplitude $[k, +\infty)$.
			\item $\pi_0(A)\otimes_A F$ is $k$-connective.
	\end{enumerate}
\end{lem}
\begin{proof}
Immediate from the Tor spectral sequence \cite[Corollary 7.2.1.23]{Lurie_Higher_algebra}.
\end{proof}
}

\begin{theorem}\label{prop:perverse_flatness_reformulation}
	Let $(X,P)$ be a conically stratified space with locally weakly contractible strata and let $\pfrak \colon P \to \mathbb Z$ be a function.
	Assume that $(X,P)$ is locally categorically compact and that $P$ is finite.
	Let $A \in \dCRing$ and let $F \in \Conshyp_{P}(X;\Mod_A)$. 
	Then the following statements are equivalent:
	\begin{enumerate}\itemsep=0.2cm
		\item The hypersheaf $F$ is $\ptauhypX$-flat relative to $A$.
		
		\item For each $p \in P$, the stalks of $i_p^{\ast,\hyp}(F)$ have $ \Tor $-amplitude contained in $[\pfrak(p),+\infty)$ (equivalently, they are $\pfrak(p)$-connective) and the stalks of $i_p^{!,\hyp}(F)$ have $ \Tor $-amplitude contained in $(-\infty,\pfrak(p)]$.
		\personal{Mauro: it's really weird to say that something has $ \Tor $-amplitude in $[a,+\infty)$, as this is really equivalent to $a$-connectivity. On the other hand, having $ \Tor $-amplitude in $(-\infty,a]$ is very different from $a$-coconnectivity.}
	\end{enumerate}
\end{theorem}

\begin{proof}
     By \cref{lem:checking_perverse_on_stalks}, the hypersheaf  $F$ is $\ptauhypX$-flat relative to $A$ if and only if for every $M \in \Mod_A^\heartsuit$ and every $p\in P$,  the stalks of $i_p^{\ast,\hyp}(M\otimes_A^{\hyp} F)$ are $\pfrak(p)$-connective and the stalks of $i_p^{!,\hyp}(M \otimes_A^{\hyp} F)$ are $\pfrak(p)$-coconnective.
     By \cref{hyp_coeff_change_inverse_image}, we have
	\[ i_p^{\ast,\hyp}(M\otimes_A^{\hyp} F)\simeq M\otimes_A^{\hyp} i_p^{\ast,\hyp}(F) \ . \] 
     Furthermore, \cref{shriek_and_tensor_product} gives
	\[ i_p^{!,\hyp}(M\otimes_A^{\hyp} F)\simeq M \otimes_A^{\hyp} i_p^{!,\hyp}(F) \ . \]
     The conclusion thus follows.
\end{proof}

\begin{rem}
	It should be noted that, although the proof is simple, the whole difficulty is hidden in \cref{shriek_and_tensor_product}, which ultimately relies on the key computation performed in \cite[Proposition 6.7.5]{Porta_Teyssier_Exodromy}.
	The proof of the latter result relies on a careful analysis of the geometry of conically stratified spaces.
\end{rem}

The following is the translation of \cite[Lemma II.2.46]{DPS_Stable_pairs} in our setting:

\begin{cor}\label{tau_flatness_flat_hyperdescent}
	In the setting of \cref{prop:perverse_flatness_reformulation}, let $A \to B$ be a faithfully flat map in $\dCRing$.
	Then $F \in \Conshyp_P(X;\Mod_A)$ is \smash{$\ptauhypX$}-flat relative to $A$ if and only if $B \otimes_A^{\hyp} F$ is \smash{$\ptauhypX$}-flat relative to $B$.
\end{cor}
\begin{proof}
	The ``only if'' is the content of \cref{tau_flatness_stable_under_pullback}.
	By \cref{prop:perverse_flatness_reformulation}, \cref{hyp_coeff_change_inverse_image} and \cref{prop:uppershriek_and_constructibility}-\eqref{prop:uppershriek_and_constructibility:uppershriek_change_of_coefficients}, the ``if'' reduces to the fact that $ \Tor $-amplitude of given range is a fpqc local property \cite[Proposition 2.8.4.2]{Lurie_SAG}.
\end{proof}

\section{Perverse character stacks}\label{moduli}

We approach the main representability result.
Its proof combines the finiteness theorem from \cite{Beyond_conicality} (see Theorems \ref{thm:subanalytic_stratified_spaces_are_conically_refineable} \& \ref{thm:algebraic_stratified_spaces_are_conically_refineable_and_categorically_finite}) with the characterization of perverse flatness of \cref{prop:perverse_flatness_reformulation}.

\begin{convention}
	Throughout this whole section, we fix $k \in \dCRing$.
	All derived stacks should tacitly be understood to be over $k$.
\end{convention}


\subsection{Moduli of constructible complexes}

We first construct a derived geometric stack, locally of finite type, parametrizing complexes of constructible sheaves on a fixed (sufficiently nice) stratified space $(X,P)$.
This is not a $1$-Artin stack, but rather a higher stack, in the sense of Simpson \cite{Simpson_Algebraic_1996}.
Following the general philosophy of \cite[\S II.2]{DPS_Stable_pairs}, using the perverse $t$-structure, we cut out the $1$-Artin stack of perverse sheaves as an open in the larger stack of complexes of constructible sheaves.

\begin{notation}
	Given a stratified space $ (X,P) $ and $A \in \dCRing_k$, we write 
	\begin{equation*}
		\Conshyp_{P,\omega}(X; \Mod_A) \subset \Conshyp_{P}(X; \Mod_A)
	\end{equation*}
	for the full subcategory spanned by the hyperconstructible hypersheaves on $ (X,P) $ whose stalks belong to $\Perf_A$.
\end{notation}

\begin{defin}
	Let $(X,P)$ be a stratified space.
	Let  
	\[ \bfConsPhyp(X) \colon \dAff_k\op \to \cS \]
	be the derived prestack sending $S = \Spec(A) \in \dAff_k$ to the maximal $\infty$-groupoid contained in $\Conshyp_{P,\omega}(X; \Mod_A)$.
	The functoriality is induced by the extension of scalars, discussed in \cref{projection_formula}.
	When $P = \ast$, we write $\bfLoc(X)$ instead of $\bfCons_\ast(X)$.
\end{defin}

\begin{rem}
	Given a map of stratified spaces $f \colon (X,P) \to (Y,Q)$, the functor $f^{\ast,\hyp}$ of \cref{rem:pullback_functoriality} induces a morphism of derived prestacks $\bfConsQhyp(Y)\to \bfConsPhyp(X)$.
\end{rem}

\begin{lem}\label{lem:descent_ConsP_X}
	Let $(X,P)$ be a  stratified space and let $k$ be a derived commutative ring.
	Then, the presheaf 
	 \[ \bfConsPhyp(-) \colon \Open(X)\op \to \PSh(\dAff_k) \]
 	satisfies hyperdescent.
\end{lem}

\begin{proof}
     Since constructibility and compactness of stalks can be checked locally, the statement follows from hyperdescent of $\HSh(-,\cE)$, which holds for every $\cE \in \PrL$.
\end{proof}

Our first task is to compare the derived prestack \smash{$\bfConsPhyp(X)$} with the \emph{moduli of objects} of the (large) $\infty$-category $\ConsPhyp(X;\Mod_k)$.
For the convenience of the reader, we briefly recall the Toën and Vaquié's construction (see \cite[\S II.2.4]{DPS_Stable_pairs} for a more extended survey):

\begin{recollection}[Toën and Vaquié's moduli of objects]\label{recollection_Toen_Vaquie}
	Let $\cC\in \PrLomega_k$.
	Its \emph{moduli of objects} is the derived stack
	\[ \cM_{\cC}\colon \dAff_k\op \to  \cS \]
	given by the rule
	\[ \cM_{\cC}(\Spec(A)) \coloneqq \Fun^{\st}_k((\cC^{\omega})\op, \Perf(A))^\simeq \ , \]
	where $\Fun^{\st}_k$ denotes the $\infty$-category of exact, $k$-linear functors.
	When $\cC$ is a compact object of $\PrLomega_k$,
	\cite[Theorem 3.6]{Toen_Moduli} states that $\cM_{\cC}$ is a locally geometric derived stack locally of finite presentation.
	Moreover, Corollary 3.17 in \emph{loc.\ cit.} states that if $x \colon \Spec(A) \to \cM_\cC$ classifies $F \colon (\cC^\omega)\op \to \Perf(A)$, then the pullback of the tangent complex of $\cM_\cC$ at $x$ is given by the formula
	\[ x^\ast \mathbb T_{\cM_\cC} \simeq \Hom_{\Fun^{\st}_k((\cC^\omega)\op,\Perf(A))}(F,F)[1] \ . \]
\end{recollection}


\begin{lemma}\label{exodromy_compact_objects}
	Let $(X,P)$ be an exodromic stratified space with locally weakly contractible strata.
	Let $A \in \dCRing_k$. Then:
	\begin{enumerate}\itemsep=0.2cm
		\item The exodromy equivalence of \cref{thm:exodromy_with_coefficients} restricts to an equivalence
		\[ \Fun(\Piinfty(X,P), \Perf_A) \simeq \Conshyp_{P,\omega}(X;\Mod_A) \ . \]
		
		\item The derived prestack $\bfConsPhyp(X)$ satisfies flat hyperdescent and it is thus a derived stack.
	\end{enumerate}
\end{lemma}

\begin{proof}
	The second statement follows from the first and the fact that the assignment $A \mapsto \Perf_A$ satisfies flat hyperdescent \cite[Corollary D.6.3.3 \& Proposition 2.8.4.2-(10)]{Lurie_SAG}.
	The first statement follows from the fact that, in the exodromy equivalence
	\begin{equation*}
		\Fun(\Piinfty(X,P), \Mod_A) \simeq \ConsPhyp(X;\Mod_A) \ ,
	\end{equation*}
	evaluation at an object $[x] \in \Piinfty(X,P)$ corresponds to taking the stalk at a point $x \in X$ representing $[x]$, see \cref{thm:exodromy_functoriality} and \cref{obs:objects_of_Exit_with_locally_weakly_contractible_strata}.
\end{proof}

\begin{lem}\label{toen_vaquié_comparaison}
    Let $ (X,P) $ be an exodromic stratified space.
    Set $\cC \coloneqq \ConsPhyp(X;\Mod_k)$. 
	Then for every $A \in \dCRing_k$, there is a canonical equivalence
	\begin{equation}\label{eq_toen_vaquié_comparaison}
		\Fun_k^{\mathrm{st}}\big( (\cC^\omega)\op, \Mod_A \big)   \simeq \Conshyp_{P}(X;\Mod_A)  \ .
	\end{equation}
\end{lem}

\begin{proof}
	Note that $\cC$ is a compactly generated presentable stable $k$-linear $\infty$-category in virtue of \cref{compact_generators_Cons}.
	We have:
	\begin{align*}
		\ConsPhyp(X;\Mod_A) & \simeq \ConsPhyp(X) \otimes \Mod_A \\
		& \simeq (\ConsPhyp(X) \otimes \Mod_k) \otimes_{\Mod_k} \Mod_A \\
		& \simeq \cC \otimes_{\Mod_k} \Mod_A \\
		& \simeq \Fun_k^{\mathrm{R}}(\cC\op, \Mod_A) \\
		& \simeq \Fun_k^{\st}((\cC^\omega)\op, \Mod_A) \ ,
	\end{align*}
    and the conclusion follows.
\end{proof}

\begin{lem}\label{toen_vaquié_comparaison_compact}
	In the setting of \cref{toen_vaquié_comparaison}, assume furthermore that $(X,P)$ has locally weakly contractible strata.
	Then the equivalence \eqref{eq_toen_vaquié_comparaison} restricts to an equivalence
	\[ \Fun_k^{\mathrm{st}}\big( (\cC^\omega)\op, \Perf(A)\big)\simeq \Conshyp_{P,\omega}(X;\Mod_A) \ . \]
	In particular, there is a canonical equivalence of derived stacks
	\[ \cM_{\cC} \simeq \bfConsPhyp(X) \ . \]
\end{lem}

\begin{proof}
	Under the equivalence \eqref{eq_toen_vaquié_comparaison}, the hypersheaf $F\in \Conshyp_{P}(X;\Mod_A)$ corresponds to the assignment
	\[ G \longmapsto \Hom_{ \Conshyp_{P}(X;\Mod_A)}(A\otimes_k^{\hyp} G,F) \ . \]
	Note also that for a functor in $\Fun_k^{\mathrm{st}}\big( (\cC^\omega)\op, \Mod_A \big)$, the condition to land in $\Perf(A)$ can be checked on a subset of objects generating $\cC^\omega$ under finite colimits and retracts.
	Hence, the condition to land in $\Perf(A)$ can be checked on a set of compact generators of $\cC$.
	Since $(X,P)$ has locally weakly contractible strata, \cref{compact_generators_Cons} implies that  $\cC$ is compactly generated by the $x_{\sharp}^{\cons}(k)$ for $x\in X$.
	On the other hand, we have 
   	\begin{align*}
		\Hom_{ \Conshyp_{P}(X;\Mod_A)}(A\otimes_k^{\hyp} x_{\sharp}^{\cons}(k),F)& \simeq  \Hom_{ \Conshyp_{P}(X;\Mod_A)}(x_{\sharp}^{\cons}(A),F)   \\ 
                       &   \simeq \Hom_{ \Mod_A}(A,x^{\ast,\hyp}(F)) \\
                       & \simeq x^{\ast,\hyp}(F)\ .
    \end{align*}
	The conclusion thus follows.  
\end{proof}

In order to invoke Toën and Vaquié's representability result, one needs a compactness assumptions on \smash{$\ConsPhyp(X;\Mod_k)$}.
This are provided by \Cref{thm:subanalytic_stratified_spaces_are_conically_refineable,thm:algebraic_stratified_spaces_are_conically_refineable_and_categorically_finite} via:

\begin{lem}\label{compactness_criterion_for_PSh}
	Let $\cK$ be the collection of finite simplicial sets with $\Idem$ added  and let
	\begin{equation*}
		\PSh_{\cK} \colon \Cat_{\infty}\to \Cat_\infty^{\mathrm{rex},\mathrm{idem}}
	\end{equation*}
	be the left adjoint of the forgetful functor $\Cat_\infty^{\mathrm{rex},\mathrm{idem}}\to \Cat_{\infty}$.
	(See \cite[Proposition 5.3.6.2]{HTT} for why this adjoint exists.)
	Then, for every $\cC \in \Cat_{\infty}$, the following hold:
	\begin{enumerate}\itemsep=0.2cm
		\item We have $\PSh_{\cK}(\cC)=\PSh(\cC)^{\omega}$.

		\item If $\cC$ is compact, then $\PSh(\cC)$ is a compact object in $\PrLomega$.

		\item If $\cC$ is compact, the tensor product
		\[ \PSh(\cC)_k \coloneqq\PSh(\cC) \otimes \Mod_k \]
		is compact in $\PrLomega_k$.
	\end{enumerate}
\end{lem}

\begin{proof}
	By construction, $\PSh_{\cK}(\cC)$ is the smallest full subcategory of $\PSh(\cC)$ containing the image of the Yoneda embedding and closed under $\cK$-indexed colimits. In particular, $\PSh_{\cK}(\cC) \subseteq \PSh(\cC)^{\omega}$ and the equality holds thanks to \cite[Proposition 5.3.4.17]{HTT}; this proves (1).
	
	For (2), assume that $\cC$ is compact in $\Cat_{\infty}$.
	By \cite[Proposition 5.5.7.11 \& Corollary 4.4.5.21]{HTT}, the right adjoint $\Cat_\infty^{\mathrm{rex},\mathrm{idem}}\to \Cat_\infty$ commutes with filtered colimits.
	Thus, $\PSh_{\cK}(\cC)$ is  a compact object of $\Cat_\infty^{\mathrm{rex},\mathrm{idem}}$.
	Since $\Ind$ induces an equivalence between $\Cat_\infty^{\mathrm{rex},\mathrm{idem}}$ and $\PrLomega$ by  \cite[Lemma 5.3.2.9]{Lurie_Higher_algebra}, we deduce that $\Ind(\PSh_{\cK}(\cC))$ is compact in $\PrLomega$.
	By (1) combined with the fact that $\PSh(\cC)$ is compactly generated in virtue of \cite[Proposition 5.3.5.12]{HTT}, we have 
	\[ \Ind(\PSh_{\cK}(\cC)) \simeq \Ind(\PSh(\cC)^{\omega}) \simeq \PSh(\cC) \]
	and (2) is proved.
	
	Item (3) follows from (2) and the fact that the forgetful functor $\PrLomega_k \to \PrLomega$ commutes with limits and colimits.
	Hence, its left adjoint preserves compact objects.
\end{proof}

\begin{cor}\label{ConsP_is_compact}
    Let $(X,P)$ be a categorically compact exodromic stratified space.
	Then, \smash{$\ConsPhyp(X;\Mod_k)$} is compact in $\PrLomega_k$.
\end{cor}

\begin{proof}
	The exodromy equivalence yields 
	\[ \ConsPhyp(X;\Mod_k) \simeq \Fun\big(\Pi_\infty(X,P), \Mod_k) \simeq  \PSh\big(\Pi_\infty(X,P )\op\big) \otimes \Mod_k \ . \]
	Thus the claim follows from \cref{compactness_criterion_for_PSh}.
\end{proof}

Summing up the results discussed so far, we obtain:

\begin{prop} \label{prop:categorically_compact_implies_representable}
	Let $(X,P)$ be a categorically compact exodromic stratified space with locally weakly contractible strata.
	Then:
	\begin{enumerate}\itemsep=0.2cm
		\item The derived stack $\bfConsPhyp(X)$ is locally geometric and locally of finite presentation.

		\item The tangent complex at a point $x \colon \Spec(A) \to \bfConsPhyp(X)$ classifying a constructible sheaf $F \in \Conshyp_{P,\omega}(X;\Mod_A)$ is given by
		\[ x^\ast \bbT_{\bfConsPhyp(X)} \simeq \Hom_{\ConsPhyp(X;\Mod_A)}(F, F)[1] \ , \]
		where the right hand side denotes the $\Mod_A$-enriched $\Hom$ of $\bfConsPhyp(X;\Mod_A)$.
	\end{enumerate}	
\end{prop}

\begin{proof}
	Set $\cC \coloneqq \ConsPhyp(X;\Mod_k)$.
	By \cref{toen_vaquié_comparaison_compact}, the stack $\bfConsPhyp(X)$ is canonically equivalent to the moduli of objects $\cM_{\cC}$.
    Since $(X,P)$ is categorically compact, the category $\cC$ is a compact object of $\PrLomega_k$ in virtue of \cref{ConsP_is_compact}.
    Then the conclusion follows from Toën and Vaquié's results, see \cref{recollection_Toen_Vaquie}.
\end{proof}


To get rid of the dependence of the specific stratification, we study how the moduli $\bfConsPhyp(X)$ behaves under coarsenings:

\begin{proposition}\label{prop:moduli_coarsening}
	Let $ (X,R) $ be a categorically compact exodromic stratified space with locally weakly contractible strata.
	Let $ \phi \colon R \to P  $ be a map of posets.
	Then the induced map of locally geometric derived stacks 
	\begin{equation*}
		 i \colon \bfConsPhyp(X) \hookrightarrow \bfConsRhyp(X)
	\end{equation*}
	is a representable open immersion.
\end{proposition}

\begin{proof}
	From \cref{prop:categorical_finiteness_and_compactness_are_stable_under_coarsening}, we see that $ (X,P) $ is exodromic and categorically compact.
	Therefore, \cref{prop:categorically_compact_implies_representable} implies that both $\bfConsPhyp(X)$ and $\bfConsRhyp(X)$ are locally geometric and locally of finite presentation.
	In particular, the morphism $i$ automatically locally of finite presentation, and to prove that it is an open immersion suffices to prove that it is étale and that the diagonal map
	\begin{equation*}
		\Delta_i \colon \bfConsPhyp(X) \to \bfConsPhyp(X) \times_{\bfConsRhyp(X)} \bfConsPhyp(X) 
	\end{equation*}
	is an equivalence.
	\Cref{thm:stability_coarsening} shows that $\Piinfty(X,R) \to \Piinfty(X,P)$ exhibits $\Piinfty(X,P)$ as the localization of $\Piinfty(X,R)$ at the collection of morphisms $W_P$.
	It follows that for every derived $k$-algebra $A$, the functor
	\begin{equation}\label{eq:moduli_coarsening}
		\ConsRhyp(X;\Mod_A) \longrightarrow \ConsPhyp(X;\Mod_A)
	\end{equation}
	is fully faithful.
	In particular, $\bfConsPhyp(X)(\Spec(A)) \to \bfConsRhyp(X)(\Spec(A))$ is fully faithful; this immediately implies that $\Delta_i$ is an equivalence.
	\personal{Observe that this condition alone already implies that the map $ i $ is representable by (a priori non separated) algebraic spaces, because the fibers are $ 0 $-truncated and geometric.}
	
	We are left to check that $i$ is étale.
	Notice that $ i $ is automatically locally of finite presentation.
	Thus \cite[Corollary 2.2.5.6]{HAG-II} implies that it suffices to show that $ i $ is formally étale, i.e., that the cotangent complex of $ i $ vanishes.
	This follows at once from fully faithfulness of \eqref{eq:moduli_coarsening} and the formula for the tangent complex provided by \cref{prop:categorically_compact_implies_representable}-(2).
\end{proof}

We finally arrive at:

\begin{thm}\label{thm:varying_stratification}
	\hfill
	\begin{enumerate}\itemsep=0.2cm
		\item Let $X$ be a compact subanalytic space.
		Then there exists a derived stack $\bfCons(X)$ over $k$ parametrizing complexes of sheaves with perfect stalks that are constructible with respect to \emph{some} subanalytic stratification of $X$.
		
		\item Let $X$ be a $\mathbb R$-algebraic variety.
		Then there exists a derived stack $\bfCons(X)$ complexes of sheaves with perfect stalks that are constructible with respect to \emph{some} algebraic stratification of $X$.
	\end{enumerate}
	Furthermore, in both cases $\bfCons(X)$ is locally geometric and locally of finite presentation.
\end{thm}

\begin{proof}
	In both cases, we define
	\[ \bfCons(X) \coloneqq \colim_P \bfConsPhyp(X) \ , \]
	where the colimit is taken over subanalytic (resp., algebraic) stratifications of $X$.
	Using \cref{thm:subanalytic_stratified_spaces_are_conically_refineable} (resp., \ref{thm:algebraic_stratified_spaces_are_conically_refineable_and_categorically_finite}), we see that the assumptions of \cref{prop:categorically_compact_implies_representable} are satisfied.
	On the other hand, \cref{prop:moduli_coarsening} implies that the transition morphisms are representable by open Zarisi immersions, whence the conclusion.
\end{proof}


\subsection{Moduli of Perverse hyperconstructible hypersheaves}
     
We now turn our attention to the moduli of \textit{perverse} sheaves.
In light of \cref{tau_flatness_stable_under_pullback}, we can make the following definition:

\begin{defin}\label{def:pPerv_stack}
	Let $(X,P)$ be a  stratified space with $P$ finite and let $\pfrak \colon P \to \mathbb Z$ be a function.
	We let
	\[ \pPervP(X) \subset \bfConsPhyp(X) \] 
	be the sub-prestack parametrizing \smash{$\ptauhypX$}-flat hypersheaves, see \cref{notation:perverse_t_structure} and \cref{def:perverse_flatness}.
	We refer to $\pPervP(X)$ as the \textit{derived prestack of perverse hyperconstructible hypersheaves on $(X,P)$}.
\end{defin}

\begin{rem}\label{rem:Perv_1_truncated}
	In the setting of \Cref{def:pPerv_stack}, let $S = \Spec(A) \in \Aff_k$ be an \emph{underived} affine scheme and let $F, G$ be two \smash{$\ptauhypX$}-flat families of constructible sheaves over $S$.
	In particular, both $F$ and $G$ belong in $\pHSh(X;\Mod_A)$.
	The axioms of a $t$-structure therefore imply that
	\[ \Map_{\HSh(X;\Mod_A)}(F,G) \simeq \tau_{\geqslant 0} \Hom_{\HSh(X;\Mod_A)}(F,G) \]
	is discrete.
	In particular, \cite[Proposition 2.3.4.18]{HTT} implies that $\pPervP(X)$ is a $1$-truncated derived prestack.
\end{rem}

\begin{lem}\label{lem:descent_PervP_X}
	Let  $(X,P)$ be a stratified space with $P$ finite.
	Then, the presheaf 
	 \[ {}^{\pfrak}\mathbf{Perv}_P(-): \Open(X)\op \to \PSh(\dAff_k) \]
	satisfies hyperdescent.
\end{lem}
\begin{proof}
	From Lemmas \ref{tau_flatness_stable_under_pullback} and \ref{lem:descent_ConsP_X}, it is enough to check that for every $S \in \dAff_k$ and every $F \in \bfConsPhyp(X)(S)$, if there exists an open cover $\{U_i\}_{i \in I} $ of $X$ such that each restriction \smash{$F|_{U_i} \in \bfConsPhyp(U_i)(S)$} is $\tensor*[^{\pfrak}]{\tau}{^{\hyp}_{U_i}}$-flat, then $F$ is \smash{$\ptauhypX$}-flat.
	This follows from \cref{locality_t_structure}.
\end{proof}

\begin{lem}\label{Perv_is_a_stack}
	Let $(X,P)$ be a conically stratified space with locally weakly contractible strata and let $\pfrak \colon P \to \mathbb Z$ be a function.
	Assume that $(X,P)$ is locally categorically compact and that $P$ is finite.
	Then, the prestack $\pPervP(X)$ satisfies flat hyperdescent, and it is thus a derived stack.
\end{lem}
\begin{proof}
    By \cref{exodromy_compact_objects}-(2), it is enough to check that \smash{$\ptauhypX$}-flatness satisfies flat hyperdescent.
    This is the content of \cref{tau_flatness_flat_hyperdescent}.
\end{proof}

\begin{lem}\label{prestack_perv_and_refinement}
	Let $(X,R)$ be a stratified space and let $\phi \colon R \to P$ be a map of posets.
	Assume that both $R$ and $P$ are finite.
	Let $ \pfrak \colon P\to \bZ $ be any function and write $ \rfrak $ for the composite $ \pfrak \phi \colon R\to \bZ $.
	Then the  square of prestacks over $k$
	\[ \begin{tikzcd}
		{}^{\pfrak} \mathbf{Perv}_P(X)\arrow[hook]{r} \arrow[hook]{d} &{}^{\mathfrak r} \mathbf{Perv}_R(X)\arrow[hook]{d}\\
		\bfConsPhyp(X)\arrow[hook]{r}  &\bfConsRhyp(X)
	\end{tikzcd} \]
	is a pullback.
\end{lem}
\begin{proof}
	Immediate from \cref{Perv_and_refinement}.
\end{proof}

\begin{notation}\label{notation:perfect_diagrams}
	For integers $a \leqslant b$, we let 
	\[ \bfPerf_{[a,b]} \subset \bfPerf \]
	be the substack  parametrizing perfect complexes of $ \Tor $-amplitude contained in $[a,b]$.
\end{notation}

\begin{notation}
	Given an $\infty$-category $\cC$, we write $\bfPerf^\cC$ (resp., $\bfPerf^\cC_{[a,b]}$) for the derived stack sending $A \in \dAff$ to the maximal $\infty$-groupoid contained in $\Fun(\cC, \Perf(A))$ (resp., $\Fun(\cC, \Perf(A)_{[a,b]})$).
	Notice that, when $\cC = \Pi_\infty(X)$ for some locally weakly contractible topological space, we have
	\begin{equation*}
		\bfPerf^{\Pi_\infty(X)} = \bfLoc(X) \rlap{\ .}
	\end{equation*}
	In this case, we use the notation $\bfLoc(X)_{[a,b]}$.
\end{notation}

\begin{lem}\label{openness_general}
	Let $\cC$ be an $\infty$-category such that $\pi_0(\cC^{\simeq})$ is finite.
	Let $a \leqslant b$ be integers.
	Then the induced morphism of derived stacks
	\[ \bfPerf^\cC_{[a,b]} \to \bfPerf^{\cC} \]
	is representable by an open immersion.
\end{lem}

\begin{proof}
	Choose a finite set $S$ of objects of $\cC$ representing all equivalence classes of objects in $\cC$.
	Then the diagram of derived stacks
	\[ \begin{tikzcd}
		\bfPerf^\cC_{[a,b]} \arrow{r} \arrow{d} &\prod_{c \in S} \bfPerf_{[a,b]}  \arrow{d} \\
		\bfPerf^{\cC} \arrow{r} & \prod_{c \in S}  \bfPerf   \ .
	\end{tikzcd} \]
	is cartesian.
	Since $S$ is finite, \cite[Proposition 6.1.4.5]{Lurie_SAG} implies that the right vertical map is an open immersion.
	Hence, so is the left vertical map.
	\personal{Indeed, the first half of the cited proposition -- applied with $A = B$ -- guarantees that $F \in \Perf(A)$ has $ \Tor $-amplitude in $[a,b]$ if and only if for every prime ideal $\pfrak \in \Spec(\pi_0(A))$, one has that $F \otimes_A \kappa(\pfrak)$ has homological amplitude in $[a,b]$.
	The second half, implies that the set of such primes is an open subset of $\Spec(A)$.}
\end{proof}

The condition of \cref{openness_general} is automatic for compact $\infty$-categories: 

\begin{lem}\label{compact_implies_finite_pi_0}
	Write \smash{$ \Cat_{\infty}^{\pi_{0}\textup{-fin}} $} for the full subcategory of $ \Cat_{\infty} $ spanned by those $ \infty $-categories $ \cC $ for which $\pi_0(\cC^{\simeq})$ is finite.
	Then \smash{$ \Cat_{\infty}^{\pi_{0}\textup{-fin}} $} is closed under retracts, finite coproducts, and pushouts.
	In particular, \smash{$ \Cat_{\infty}^{\pi_{0}\textup{-fin}} $} contains all compact $ \infty $-categories.
\end{lem}

\begin{proof}
	First observe that since finite sets are closed under retracts in $ \dSet $, the subcategory \smash{$ \Cat_{\infty}^{\pi_{0}\textup{-fin}} $} is closed under retracts.
	Second, note that the maximal sub-$ \infty $-groupoid functor $ (-)^{\simeq} \colon \Cat_{\infty} \to \Spc $ preserves coproducts, and $\pi_0 \colon \Spc \to \dSet$ is a left adjoint, the functor $ \cC \mapsto \pi_0(\cC^{\simeq}) $ also preserves coproducts.
	Since finite sets are closed under finite coproducts, we deduce that \smash{$ \Cat_{\infty}^{\pi_{0}\textup{-fin}} $} is closed under finite coproducts.
	For closure under pushouts, notice that given a span of $ \infty $-categories $ \cB \leftarrow \cA \to \cC $, there is a surjective functor $ \cB \sqcup \cC \to \cB \sqcup^{\cA} \cC $
	from the coproduct of $ \cB $ and $ \cC $ to the pushout.
	So if $ \pi_0(\cB^{\simeq}) $ and $ \pi_0(\cC^{\simeq}) $ are finite, then $ \pi_0((\cB \sqcup^{\cA} \cC)^{\simeq}) $ is also finite.

	For the final statement, recall that the full subcategory spanned by the compact $ \infty $-categories is the smallest full subcategory containing $ \emptyset $, $ \ast $, and $ [1] $ and closed under pushouts.
\end{proof}

\begin{lem}\label{Perv_as_pullback}
	Let $(X,P)$ be a conically stratified space with locally weakly contractible strata and let $\pfrak \colon P \to \bZ$ be a function. 
	Assume that $(X,P)$ is locally categorically compact and that $P$ is finite.
	Then, the square of derived stacks 
	\[ \begin{tikzcd}[column sep=9em]
		\pPervP(X) \arrow{r}{(i_p^{\ast,\hyp}, i_p^{!,\hyp} )_{p\in P}} \arrow{d} & \prod_{p\in P} \bfLoc(X_p)_{[\pfrak(p),+\infty)}\times  \bfLoc(X_p)_{(-\infty,\pfrak(p)]}\arrow{d} \\
		\bfConsPhyp(X) \arrow[r, "{(i_p^{\ast,\hyp}, i_p^{!,\hyp} )_{p\in P}}"'] & \prod_{p\in P} \bfLoc(X_p) \times \bfLoc(X_p)
	\end{tikzcd} \]
	is a pullback.
\end{lem}

\begin{proof}
	Since $(X,P)$ is conically stratified with locally weakly contractible strata and locally categorically compact, \cref{prop:uppershriek_and_constructibility}-\eqref{prop:uppershriek_and_constructibility:preservation} ensures that the lower horizontal arrow is well-defined.
	The implication (1) $\Rightarrow$ (2) in \cref{prop:perverse_flatness_reformulation} guarantees that the upper horizontal arrow is well-defined as well.
	The implication (2) $\Rightarrow$ (1) from the same theorem implies that the above square is a pullback.
\end{proof}

\begin{theorem}\label{thm:representability_perv}
	Let $(X,R)$ be a conically stratified space with locally weakly contractible strata, let $\phi \colon R\to P$ be a map of posets, and let $ \pfrak \colon P\to \bZ $ be a function.
	Assume that $R$ and $P$ are finite and that $ (X,R) $ is locally categorically compact.
	Then:
	\begin{enumerate}\itemsep=0.2cm
		\item The derived prestack $\pPervP(X)$ satisfies flat hyperdescent and it is thus a derived stack.
	\end{enumerate}
	Assume furthermore that $(X,R)$ is categorically compact.
	Then:
	\begin{enumerate}\itemsep=0.2cm\setcounter{enumi}{1}
        \item The morphism $\pPervP(X) \hookrightarrow \bfConsPhyp(X)$ is representable by open Zariski immersions.
		
		\item The derived stack $ \pPervP(X) $ is $1$-Artin and locally of finite presentation.
	\end{enumerate}
\end{theorem}

\begin{proof}
	Write $ \rfrak \coloneqq \pfrak \phi \colon R\to \bZ $.
	For item (1), observe that $ (X,R) $ and $ (X,P) $ are exodromic thanks to \cref{eg:conical_stratifications_are_exodromic} and \cref{thm:stability_coarsening}.
	Furthermore, \cite[Lemma 5.2.8]{Beyond_conicality} guarantees that they have locally weakly contractible strata.
	By \cref{exodromy_compact_objects}-(2), the prestacks $ \bfConsPhyp(X) $ and $ \bfConsRhyp(X) $ thus satisfy flat hyperdescent.
	Moreover, \cref{Perv_is_a_stack} shows that $ \rPervR(X) $ satisfies flat hyperdescent and the conclusion follows from \cref{prestack_perv_and_refinement}.
	
	We now prove item (2).		
	By \cref{prestack_perv_and_refinement}, it is enough to show that 
	\begin{equation*} 
		\rPervR(X) \hookrightarrow \bfConsRhyp(X)
     \end{equation*}
	is representable by open Zariski immersions.
	Since $(X,R)$ is categorically compact, \cref{lem:pulling_back_to_a_locally_closed_subposet_preserves_categorical_finiteness_and_compactness} shows that for every $r\in R$, the $\infty$-groupoid $\Pi_{\infty}(X_r)$ is compact and hence $\pi_0(\Pi_\infty(X_r))$ is finite by \cref{compact_implies_finite_pi_0}.
	Hence, \cref{openness_general} implies that the right vertical arrow of the pullback square provided by \cref{Perv_as_pullback} is representable by open Zariski immersions.
	The conclusion follows.
	
	We now prove (3).
	Observe that $(X,P)$ is categorically compact by \cref{prop:categorical_finiteness_and_compactness_are_stable_under_coarsening}.
	Hence,  \cref{prop:categorically_compact_implies_representable} shows that $\bfConsPhyp(X)$ is locally geometric and locally of finite presentation.
	Thus, item (2) implies that $\pPervP(X)$ is locally geometric and locally of finite presentation, and the conclusion follows from \cref{rem:Perv_1_truncated}.
\end{proof}

\begin{rem}[The conically stratified case]\label{rem:comparison_with_DPS}
	If $(X,P)$ is conically stratified (as opposed to conically refineable, as in \Cref{thm:representability_perv}).
	Then \cref{prop:uppershriek_and_constructibility} implies the perverse $t$-structure \smash{$\ptauhypX$} descends to a $t$-structure on $\ConsPhyp(X;\Mod_k)$.
	In this situation, $\pPervP(X)$ coincides exactly with the moduli of flat objects discussed in \cite[\S II.2.6]{DPS_Stable_pairs}.
\end{rem}

Examples where $ \bfConsPhyp(X) $ and $ \pPervP(X) $ are locally geometric and locally of finite presentation come in abundance:

\begin{corollary}\label{cor:examples_when_ConsP_is_locally_geometric}
	Let $(X,P)$ be a stratified space with $P$ finite, let $\pfrak \colon P\to \bZ$ be any function, and let $ k $ be a derived commutative ring.
	Assume one of the following conditions:
	\begin{enumerate}\itemsep=0.2cm
		\item The stratified space $(X,P)$ admits a refinement by a finite triangulation.

		\item The topological space $X$ is compact and $(X,P)$ admits the structure of a subanalytic stratified space in the sense of \Cref{def:subanalytic_stratified_space}.

		\item The stratified space $(X,P)$ admits the structure of an algebraic stratified space in the sense of \Cref{def:algebraic_stratified_space}.
	\end{enumerate}
	Then the derived prestacks $\bfConsPhyp(X)$ and $\pPervP(X)$ are derived stacks that are locally geometric and locally of finite presentation.
\end{corollary}

\begin{proof}
	In light of \cref{eg:conical_stratifications_are_exodromic} and \Cref{finite_geometric_realization_simplicial_complex}, item (1) is a special case of \cref{thm:representability_perv}.
	Similarly, by \Cref{thm:subanalytic_stratified_spaces_are_conically_refineable}-(4), item (2) is a special case of (1).
	Let us now prove (3).
	Note that by \Cref{prop:categorically_compact_implies_representable} and \Cref{thm:algebraic_stratified_spaces_are_conically_refineable_and_categorically_finite}-(2)(3), the derived prestack $ \bfConsPhyp(X) $ is a derived stack that is locally geometric and locally of finite presentation.
	Moreover, since the properties of being a derived stack, being locally geometric, and being locally of finite presentation are stable under finite limits, \Cref{lem:descent_PervP_X} reduces the claim for $ \pPervP(X) $ to the case where $ X $ is affine.
	To conclude, note that \Cref{thm:algebraic_stratified_spaces_are_conically_refineable_and_categorically_finite}-(1) shows that an affine algebraic stratified space satisfies the conditions of \cref{thm:representability_perv}-(3).
	The conclusion thus follows.
\end{proof}

\begin{cor}\label{cor:varying_stratification_perverse}
	In this statement we let $\pfrak$ denote the middle perversity function.
	\begin{enumerate}\itemsep=0.2cm
		\item Let $X$ be a compact subanalytic space.
		Then there exists a derived stack $\tensor*[^{\pfrak}]{\mathbf{Perv}}{}(X)$ over $k$ parametrizing perverse sheaves that are constructible with respect to \emph{some} subanalytic stratification of $X$.
		
		\item Let $X$ be a $\mathbb R$-algebraic variety.
		Then there exists a derived stack$\tensor*[^{\pfrak}]{\mathbf{Perv}}{}(X)$ parametrizing perverse sheaves that are constructible with respect to \emph{some} algebraic stratification of $X$.
	\end{enumerate}
	Furthermore, in both cases $\tensor*[^{\pfrak}]{\mathbf{Perv}}{}(X)$ is $1$-Artin and locally of finite presentation.
\end{cor}

\begin{proof}
	Immediate from \cref{thm:varying_stratification} and \cref{cor:examples_when_ConsP_is_locally_geometric}.
\end{proof}


\subsection{Application: the perverse cohomological Hall algebra}

To showcase the utility of the derived structure we obtained on $\pPervP(X)$, we construct a new example of a cohomological Hall algebra.
Let $X$ be a smooth curve over $\C$ equipped with $n$ distinct marked points $\{p_1, p_2, \ldots, p_n\}$.
We see $X$ as a stratified space over the poset $P=\{0 < 1\}$, equipped with the middle perversity function, that is 
\[ \pfrak(0) = 0 \ , \qquad \pfrak(1) = -1 \ . \]
Since $(X,P)$ is conically stratified, $\pPervP(X)$ is identified with the moduli of flat objects \cite[\S II.2.6]{DPS_Stable_pairs} associated to the stable $\infty$-category $\ConsPhyp(X;\Mod_\C)$ and the $t$-structure $\ptauhypX$, see \cref{rem:comparison_with_DPS}.
In particular, we obtain:

\begin{thm}\label{thm:perverse_CoHA}
	In the above situation, the Borel--Moore homology
	\begin{equation*}
		\mathrm{H}^{\mathrm{BM}}_\ast( \pPervP(X); \Q )
	\end{equation*}
	has a natural associative product, given by Hall convolution.
	A similar statement holds for $G$-theory in place of Borel--Moore homology.
	At the categorified level, $\mathrm{Coh}^b(\pPervP(X))$ supports an $\mathbb E_1$-monoidal structure given by Hall convolution.
\end{thm}

\begin{proof}
	We only need to apply \cite[Theorem II.4.3]{DPS_Stable_pairs}.
	Concerning Assumption C in \emph{loc.\ cit.}: C.1 is provided by \cref{thm:algebraic_stratified_spaces_are_conically_refineable_and_categorically_finite}; C.2 is automatic in this setting; C.3 has been verified in \cref{thm:representability_perv}-(2).
	Concerning C.4, we observe that, in the proof of \cite[Theorem II.4.3]{DPS_Stable_pairs} it is only needed its consequence \cite[Lemma II.4.2]{DPS_Stable_pairs}.
	We prove an analogue of this statement in our setting in \cref{perverse_hom_amplitude} below.
	Finally, to check that the map \cite[II.4.2]{DPS_Stable_pairs} is locally representable by proper algebraic spaces, we invoke \cite[Propositions 3.4.6 \& 4.2.8]{Lampetti_Good_moduli}.
\end{proof}

\begin{lem}\label{perverse_hom_amplitude}
	Let $k$ be a field.
	Let $(X,P)$ as above.
	Let $F,G\in \tensor*[^{\pfrak}]{\mathrm{Perv}}{_{P}^{\hyp}}(X;\Mod_k)$. 
	Then 
	\[ \Hom_{\ConsPhyp(X;\Mod_k)}(F, G) \in \Mod_k \]
	is concentrated in homological degrees $[-2,0]$.
\end{lem}

\begin{proof}
	Let us denote by
	\begin{equation*}
		\mathbb{D} \colon \ConsPhyp(X;\Mod_k)\to \ConsPhyp(X;\Mod_k)
	\end{equation*}
	the Verdier duality functor.
	Then, there is a canonical equivalence 
	\[ \Hom_{\ConsPhyp(X;\Mod_k)}(F, G)\simeq R\Gamma(X,\mathbb{D}(F\otimes \mathbb{D}G)) \ . \]
	By Poincaré--Verdier duality, showing that the above complex is concentrated in homological degrees $[-2,0]$ amounts to showing that $R\Gamma_c(X,F\otimes \mathbb{D}G)$ is concentrated in homological degrees $[0,2]$.
	Since $\mathbb{D}G$ is perverse, we are thus left to show that if $F,G\in {}^{\pfrak} \mathrm{Perv}_P^{\hyp}(X;\Mod_k)$, then $\Gamma_c(X,F\otimes G)$  is concentrated in homological degrees $[0,2]$.
	By definition, $F\otimes G$ is concentrated in homological degrees $[0,2]$ and both $\pi_{1}(F\otimes G)$ and $\pi_{0}(F\otimes G)$ are punctually supported.
	Since $X$ is a curve over $\C$, the only potential non zero terms of the spectral sequence 	
	\[ E^2_{p,q} \coloneqq \pi_p \Gamma_c(X,\pi_q (F\otimes G)) \Rightarrow \pi_{p+q}\Gamma_c(X, F\otimes G) \]
	are thus
	\[ E^2_{0,0}, \ E^2_{0,1} , \ E^2_{0,2} , \ E^2_{-1,2} , \ E^2_{-2,2} \ . \]
	This implies the sought-after range. 	
\end{proof}



\newcommand{\etalchar}[1]{$^{#1}$}
\def\cprime{$'$}
\providecommand{\bysame}{\leavevmode\hbox to3em{\hrulefill}\thinspace}
\providecommand{\MR}{\relax\ifhmode\unskip\space\fi MR }
\providecommand{\MRhref}[2]{%
	\href{http://www.ams.org/mathscinet-getitem?mr=#1}{#2}
}
\providecommand{\href}[2]{#2}

\end{document}